\documentclass[12pt]{article}

\usepackage{amsmath}
\usepackage{amsthm}
\usepackage{amssymb}
\usepackage{url}
\usepackage[multiple]{footmisc}
\usepackage{graphicx}
\usepackage{epstopdf}
\usepackage{chngcntr}
\usepackage{enumitem}
\usepackage{hhline}
\usepackage{array}
\usepackage{hyperref}
\usepackage{enumitem}
\usepackage{color}
\usepackage{subcaption}
\usepackage{float}

\hypersetup{colorlinks=false}

\usepackage[backend=biber,maxnames=10,backref=true,giveninits=true,safeinputenc]{biblatex}
\DefineBibliographyStrings{english}{%
	backrefpage = {cited on page},
	backrefpages = {cited on pages},
}

\if@twoside
\else 
\fi 

\numberwithin{equation}{section}
\counterwithin{figure}{section}
\counterwithin{table}{section}

\theoremstyle{plain}
\newtheorem{theorem}{Theorem}[section]

\newtheorem{proposition}[theorem]{Proposition}
\newtheorem{corollary}[theorem]{Corollary}

\theoremstyle{definition}

\newtheorem{assumption}[theorem]{Assumption}

\theoremstyle{remark}
\newtheorem*{remark}{Remark}

\newcommand{\norm}[1]{\left\|#1\right\|}
\newcommand{\abs}[1]{\left\vert#1\right\vert}
\newcommand{\spr}[1]{\left\langle\,#1\,\right\rangle}
\newcommand{\kl}[1]{\left(#1\right)}
\newcommand{\Kl}[1]{\left\{#1\right\}}

\newcommand{\R}{\mathbb{R}}

\newcommand{\grad}{\nabla}

\renewcommand{\div}[1]{\operatorname{div}\left(#1\right)}

\newcommand{\eps}{\varepsilon}

\newcommand{\HoO}{{H^1(\Omega)}}

\newcommand{\LtO}{{L^2(\Omega)}}
\newcommand{\LiO}{{L^\infty(\Omega)}}

\newcommand{\intVx}[1]{\int\limits_\Omega #1 \, d\mathbf{x}}

\newcommand{\bO}{{\partial \Omega}}

\newcommand{\x}{\textbf{x}}
\newcommand{\uf}{\mathbf{u}}
\newcommand{\vf}{\mathbf{v}}
\newcommand{\hf}{\mathbf{h}}
\newcommand{\wf}{\mathbf{w}}

\newcommand{\RE}{\mathcal{R}}

\renewcommand{\S}{\mathcal{S}}

\newcommand{\ufhi}{\hat{\uf}^i}
\newcommand{\xhi}{\hat{\x}^i}

\newcommand{\at}{\tilde{a}}
\newcommand{\bt}{\tilde{b}}
\newcommand{\ct}{\tilde{c}}

\newcommand{\ufb}{\bar{\uf}}

\title{Displacement field estimation from OCT images utilizing speckle information with applications in quantitative elastography}
\author{
Ekaterina Sherina\footnote{University of Vienna, Faculty of Mathematics, Oskar Morgenstern-Platz 1, 1090 Vienna, Austria (ekaterina.sherina@univie.ac.at), joint corresponding author together with Lisa Krainz.},
Lisa Krainz\footnote{Center for Medical Physics and Biomedical Engineering, Medical University of Vienna, W\"ahringer G\"urtel 18-20, A-1090 Vienna, Austria, (lisa.krainz@meduniwien.ac.at), joint corresponding author together with Ekaterina Sherina.},
Simon Hubmer\footnote{Johann Radon Institute Linz, Altenbergerstra{\ss}e 69, A-4040 Linz, Austria, (simon.hubmer@ricam.oeaw.ac.at),},\\
Wolfgang Drexler\footnote{Center for Medical Physics and Biomedical Engineering, Medical University of Vienna, W\"ahringer G\"urtel 18-20, A-1090 Vienna, Austria, (wolfgang.drexler@meduniwien.ac.at)},
Otmar Scherzer\footnote{University of Vienna, Faculty of Mathematics, Oskar Morgenstern-Platz 1, 1090 Vienna, Austria (otmar.scherzer@univie.ac.at)} \footnote{Johann Radon Institute Linz, Altenbergerstra{\ss}e 69, A-4040 Linz, Austria (otmar.scherzer@ricam.oeaw.ac.at)}
}

\begin{document}

\maketitle
%

\begin{abstract}
In this paper, we consider the problem of estimating the internal displacement field of an object which is being subjected to a deformation, from Optical Coherence Tomography (OCT) images before and after compression. For the estimation of the internal displacement field we propose a novel algorithm, which utilizes particular speckle information to enhance the quality of the motion estimation. We present numerical results based on both simulated and experimental data in order to demonstrate the usefulness of our approach, in particular when applied for quantitative elastography, when the material parameters are estimated in a second step based on the internal displacement field. 

\smallskip
\noindent \textbf{Keywords.} Displacement Field Estimation, Optical Coherence Tomography, Optical Flow Estimation, Speckle Tracking, Quantitative Elastography
\end{abstract}


\section{Introduction}

The present paper is concerned with \emph{Quantitative Elastography} from \emph{Optical Coherence Tomography (OCT)} measurements. There one tries to estimate various material parameters of a physical sample from its behaviour under deformations. This is of particular interest in medicine, where for example one tries to identify malignant tissues inside the human skin by reconstructing the spatially varying elastic parameters.

Various approaches to quantitative elastography have been proposed over the years; see e.g.\ \cite{Doy12,KenKenSam2014,LiWijDanSamMunKenObe2016,LiWijSamKenMunObe2019,ManOliDreMahKru01,Schm98} and the references therein. These differ for example in the employed deformation (e.g., quasi-static, harmonic, transient), the underlying elasticity model (e.g., linear, viscoelastic, hyperelastic), the way the deformation is measured (only on the boundary, or everywhere inside the sample using, e.g., ultrasound, magnetic resonance, or optical imaging techniques), and whether a one-step approach (i.e., reconstructing the material parameters directly from the internal images) or a two-step approach (i.e., first estimating the displacement field from the images and then computing the material parameters) is used.

In this paper, we consider in particular the displacement estimation step in two-step approaches to (quantitative) OCT elastography; see e.g.\ \cite{AdiLiaKenJohSamBop2010,KenKenSam2014,LiWijDanSamMunKenObe2016,LiWijSamKenMunObe2019,Schm98,WanLar15,WijKenSam2020} and the references therein. In some cases, the axial component of the displacement field can be measured directly via phase-sensitive OCT, given that this modality is available. However, this is usually not enough to obtain quantitative material parameter reconstructions, since important information is also contained in the other field components. Hence, the full internal displacement field is typically reconstructed via various motion estimation techniques. A standard technique in this context is the normalized cross-correlation method based on so-called \emph{speckle} in the OCT images \cite{DunKir01,RogPatFujBre04,Schm98,SunStaYan11}.

Since the early stages of the development of scattering imaging modalities, researchers are aware of the phenomenon known as speckle. It is commonly observed in ultrasound, OCT, radar-based imaging, and radio-astronomy. Speckle is usually interpreted as the constructive and destructive interference of backscattered waves, which creates the granulated appearance (produced by bright and dark spots) of OCT scans (see Figure~\ref{fig_OCT_speckle}). It was noted that multiple factors contribute to the pattern of speckles, such as properties of the optical materials or the settings of the imaging system. Although often seen as a cause of noise in OCT data, speckles also carry important information, such as on the movement of particles inside the sample. All of these topics caused a broad discussion in the literature, see for example \cite{SchmXiaYun99}. Moreover, virtual \cite{GlaSchWid15, SchmZabWidGlaLiu15} or extra speckle \cite{Schm98} has been introduced to support the motion detection.

\begin{figure}[h!] 
	\centering
	\includegraphics[width = \textwidth, clip=true, trim = {4cm 5.5cm 2cm 4cm}]{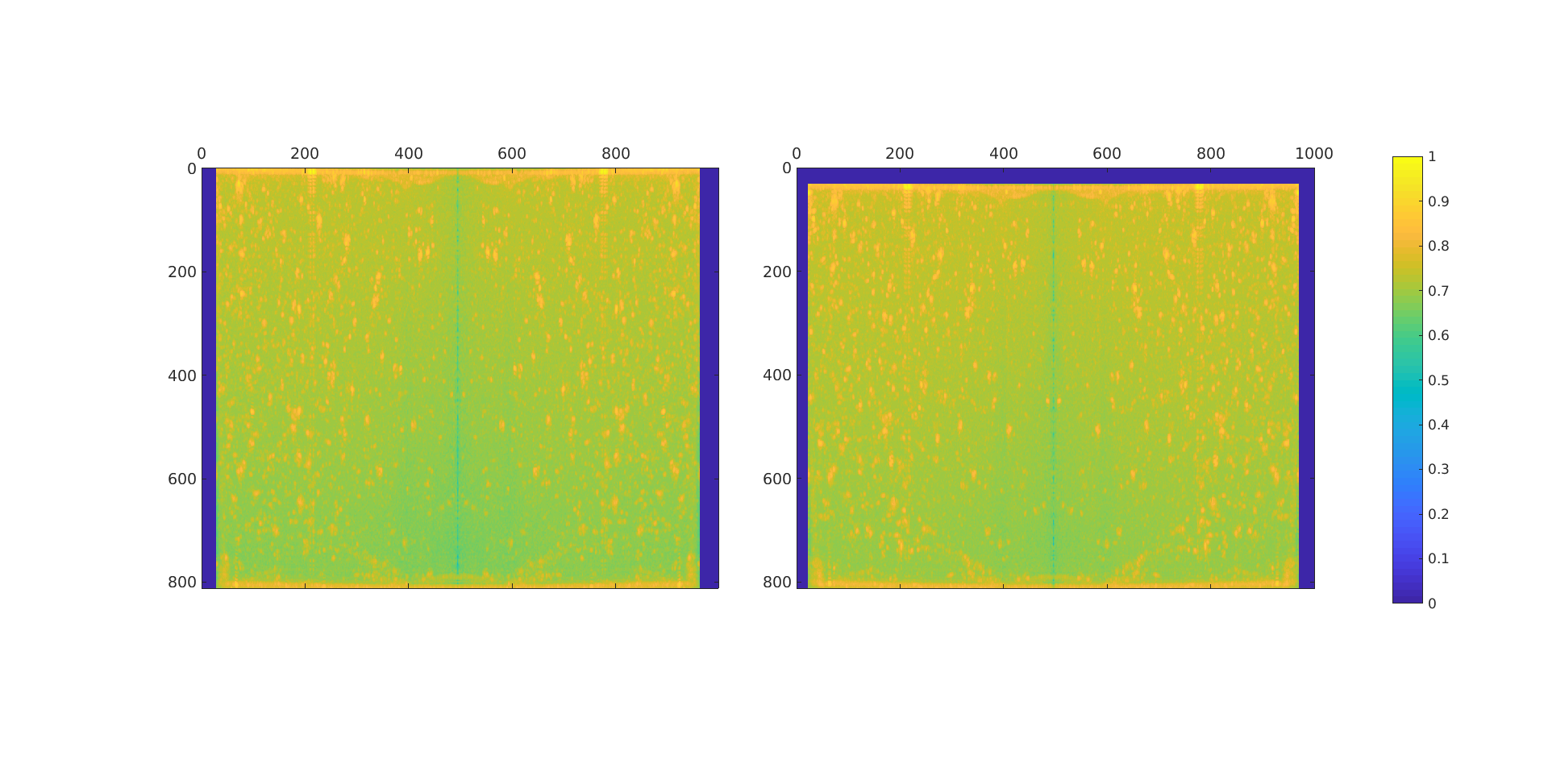}
	\quad
	\caption{Rotational maximum-intensity projection (MIP) of a volumetric OCT scan of a homogeneous cylindrical sample before (left) and after (right) compression.}
	\label{fig_OCT_speckle}
\end{figure}

The normalized cross-correlation method correlates OCT scans of pre-compressed and post-compressed media by calculating the cross-correlation coefficients for every pixel over a fixed surrounding area, and then creates a map of displacements from the maxima of those coefficients \cite{DunKir01,RogPatFujBre04,Schm98,SunStaYan11}. Unfortunately, this method is prone to miss-estimations of the displacement field if the applied strain is too small or too high, or if the size of the correlation area is not carefully chosen. Furthermore, the pixel-by-pixel processing is very time-consuming.

Hence, in this paper, we propose an alternative approach (see Sections~\ref{sect_bubble} and \ref{sect_opt_flow} below) to the cross-correlation method, which is computationally more efficient. The idea is to use the additional information contained in speckles present in OCT images in order to achieve an improved displacement field estimation. In particular, we propose a novel algorithm for detecting and tracking larger speckle formations, which we call \emph{bubbles}, and to extract the information on the general internal motion which they contain. Then, we use this information to enhance the displacement field reconstruction procedure, which is based on the OCT images. This can be interpreted as a ``large small details'' strategy. The estimated displacement fields can afterwards be used as input for various material parameter reconstruction methods in quantitative elastography, one of which we discuss below.

The outline of this paper is as follows. In Section~\ref{sect_bubble}, we introduce and discuss a method for the detection and tracking of bubbles, and in Section~\ref{sect_opt_flow}, we propose an algorithm for enhancing the displacement field estimation with the obtained bubble information. Furthermore, we present a mathematical analysis of the proposed approach. In Section~\ref{sect_mini}, we discuss possible ways to implement the proposed algorithm, and in Section~\ref{sect_numerics}, we present numerical results on both simulated and experimental data, demonstrating its practical usefulness. {In Section~\ref{sect_application}, we discuss the application of our proposed displacement field estimation algorithm to quantitative elastography, and  in Section~\ref{sect_numerics_II}, we provide numerical examples on the reconstruction of elastic parameters based on the estimated displacement fields, again on both simulated and experimental data, before proceeding to a short conclusion in Section~\ref{sect_conclusion}.

\section{A Bubble Approach to Speckle Tracking}\label{sect_bubble}

In this section, we propose a novel, heuristics-based image-processing algorithm for the detection of large speckle formations, which we call \emph{bubbles}, and the determination of their movement, from a pair of successive volumetric OCT data. The bubble information obtained from this \emph{bubble tracking} algorithm is then used in Section~\ref{sect_opt_flow} below to enhance the reconstruction quality of the displacement field estimation.

%

Our proposed method consists of the following three general steps: Pre-processing of the OCT-scan image data, detection of bubbles in the processed images, and estimation of the movement of the bubbles. These three general steps are described in detail below. Since the experimental data considered in Section~\ref{sect_numerics} stems from a  rotationally invariant cylindrical sample, the following description of the algorithm steps is tailored to this case. Note, however, that the algorithm can be generalized in a straightforward manner to non-symmetric samples.

\subsection*{Step I: Pre-processing of the OCT-scan image data}

For the pre-processing step, we start by taking a pair of volumetric OCT scans recorded from two sequential compression states of the sample. These data, which usually span several orders of magnitude, are commonly displayed in a base-$10$ logarithmic scale to enhance the visibility of lower refractions. For our purpose, we take the logarithmized data and rescale it such that its brightness lies within the interval $[0,1]$. Since we consider a rotationally invariant cylindrical sample here, we search for its center and radius through lateral segmentation and circle fitting. Using this information, the sample is aligned in the x0y plane such that its center is located at the same point in both scans. The `height' of the cylinder thus aligns with the z-axis, and data extending the actual height of the cylinder is cut. Figure~\ref{fig_OCT_speckle} depicts a rotational MIP of OCT data pre-processed in the above way. The knowledge of the origin and height of the cylinder for both compression states provides us with an initial guess of the maximal possible axial and radial displacement which can have occurred within the sample.

\subsection*{Step II: Detection of bubbles in the processed images}

The next step is the detection of bubbles within the OCT data. For this, we first apply a Gaussian filter with a standard deviation of around $0.8$ to $1$ to the data, in order to reduce background intensity jumps and to slightly smooth the granulated appearance of bubbles. Based on the histogram statistic of the volumetric dataset, we define a brightness threshold of $0.5-1\%$ of the brightest pixels, which we use to binarize the volumetric data. Each bubble in the dataset thus becomes a connected group of voxels with value $1$ within a background of uniform value $0$. Identifying and grouping the connected components of those binary images, we obtain the different bubbles, of which we extract their centroids (by fitting them with ellipsoids) as well as their pixel-volumes. Any thereby detected bubble which is smaller in voxel-volume than a predefined threshold (between 80 and 100 voxel-volumes in the experimental data set considered below) was discarded. See Figure~\ref{fig_OCT_speckle_tracking} for an illustration of the outcome of the bubble detection step described above, where the procedure was applied to the OCT images depicted in Figure~\ref{fig_OCT_speckle}.

\subsection*{Step III: Estimation of the movement of the bubbles}

In the third step, we aim at estimating the motion of the detected bubbles between the two recorded volumetric images.
 
We now compare each of the detected bubbles in the first (`start') scan with every bubble detected in the second (`end') scan. For each candidate pair of two bubbles $A$ and $B$ from the `start' and `end' scan, respectively, we compute their voxel-volumes $V_A$ and $V_B$, as well as the z-coordinates of their centroids $z_A$ and $z_B$. Furthermore, we calculate the distance between their two centroids ($d_{AB}$), the distance between the cylinder axis and the `start' centroid ($d_{O_1A}$), the distance between the cylinder axis and the `end' centroid ($d_{O_2B}$), the angle between the cylinder axis and the line connecting the `start' and `end' centroids ($\alpha$), as well as the vertical angle in the x0y plane between the line passing through the cylinder origin and the `start' centroid, and the line passing through the cylinder origin and the `end' centroid ($\varphi$). 

For each pair of bubbles $(A,B)$ we check if the following basic intuitive geometric criteria are satisfied: 
 
\begin{enumerate}
	\item The volume of a bubble does not change significantly in the images between  different compression phases, i.e.,
		\begin{equation*}
			\abs{V_A - V_B} < \eps \,.
		\end{equation*}
	
	\item A bubble does not move more than the displacement (applied for deforming the sample) in axial direction $d_{\max}$ and more than the radial expansion of the sample in lateral direction, i.e.,
		\begin{align*}
			0 < d_{AB} &\le d_{\max} \,, \\
			d_{O_1A} & < d_{O_2B} \,.
		\end{align*}
	
	\item If the sample is compressed from above, then the bubbles are mainly shifted downwards (i.e., in z-direction) between the first scan and the second, i.e.,
			\begin{equation*}
				z_A < z_B \,.
			\end{equation*}
	
	\item In the case of rotational invariance, the sample experiences a side bulging, which implies that a bubble moves outwards from the cylinder axis in radial direction. Hence, one expect that there is either no tangential movement $\varphi$, or that it is at most very small, i.e.,
		\begin{equation*}
			\varphi < \varphi_{\max} \,.
		\end{equation*}
	
	\item The possible angle $\alpha$ between the cylinder axis and the line passing through the starting and end points of the tracked bubbles is limited by the strain applied to the sample, i.e.,
		\begin{equation*}
			\alpha_{\min} \le \alpha \le \alpha_{\max} \,. 
		\end{equation*}
		
\end{enumerate}

Two bubbles $A$ and $B$ are considered to match if they satisfy all of the six inequalities stated above. Hereby, the constants $\varepsilon$, $d_{\max}$, $\alpha_{\min}$, $\alpha_{\max}$, and $\varphi_{\max}$ have to be suitably chosen in dependence on the strain applied in a concrete experiment. We then get an estimate of the bubble movement by computing the vector which points from the center of mass of bubble $A$ to the center of mass of the matched bubble~$B$. Two examples of the outcome of the bubble tracking step described above are shown in Figure~\ref{fig_OCT_speckle_tracking} (right) and Figure~\ref{fig_OCT_displacement_rotation}. There, the procedure was applied to the OCT images depicted in Figure~\ref{fig_OCT_speckle}, and to the corresponding complete volumetric OCT scan, respectively.

\subsection*{}
The behaviour of the above algorithm is demonstrated in Section~\ref{sect_numerics} below. It is both robust and efficient/fast on OCT scans with bright sharp spot reflectors. In terms of implementation, note that it can be useful not to choose one fixed $\eps$ in the inequality condition above, but to separate the bubbles into subsets of  `large' and `small' bubbles, and to use a different $\eps$ for each one of these groups. Furthermore, both the bubble detection step and the bubble matching step can be carried out in parallel. Note that there is typically at most one bubble in the second scan which matches a bubble in the first scan, and that not necessarily every bubble can be matched with another one, in particular, since the number of bubbles in two consecutive OCT scans is not constant. However, this is not a problem, since not all bubbles need to be matched in order to improve the reconstruction quality of the subsequent displacement field estimation.

\section{Displacement Field Estimation with Bubble Data}\label{sect_opt_flow}
	
In this section, we present a method for incorporating bubble movement, obtained for example with the approach presented above, into the displacement field estimation. First of all, note that for estimating the displacement or motion between images, a relatively simple way, which is valid for small displacements, is the \emph{optical-flow} equation
	\begin{equation} \label{eq_opt_flow_short}
		\nabla I \cdot \uf + I_t =0 \,,
	\end{equation}
where $I = I(\x,t)$ denotes the \emph{image intensity function} and $\uf(\x) = (\uf_1(\x),\uf_2(\x))$ denotes \emph{the displacement (motion, flow) field}. Assuming that both $I$ and $\uf$ are defined for $\x \in \Omega \subset \R^2$, one of the most common approaches for estimating the displacement field $\uf$ is to minimize the functional
	\begin{equation} \label{def_J_u}
		J(\uf) := \intVx{\left(\nabla I \cdot \uf + I_t\right)^2} +  \alpha \RE(\uf)\,,
		\qquad
		\forall \, t > 0 \,,
	\end{equation}
where $\alpha \geq 0$ is a regularization parameter and 
$\RE$ is a regularization functional, used to incorporate additional assumptions on the displacement field $\uf$. Many possible choices for $\RE$ have been proposed over the years, the most known and commonly used one perhaps being
	\begin{equation*}
	\begin{split}
		\RE(\uf) &:= \norm{\nabla \uf}^2_{\LtO} \,,
	\end{split}
	\end{equation*}
which ultimately gives rise to the well-known Horn-Schunck method. It encodes a smoothness assumption on the solution and, according to \cite{Schn91a,Sny89}, is one of the only two regularization functionals which only use first-order terms and also have a physical interpretation (the other one being the regularization functional of Nagel \cite{Nag87}).  

Note that there exists a wide variety of different approaches to motion estimation, in particular for \emph{large displacement optical flow} \cite{AubDerKor99, AubKor99, AubKor06, BakSchaLewRotBla11, BlaAna96,BroMal11,BroBreMal09, KorDerAub99,PapBruBroDidWei06,SunRotBla13,WeiBruBroPap06}. In fact, the optical flow equation \eqref{eq_opt_flow_short} is itself a first-order approximation to the more general \emph{image registration problem}, and it is valid for small displacements. Fortunately, the displacements considered throughout this paper are all comparatively small, and thus we can work with the (simpler) optical flow equation \eqref{eq_opt_flow_short}. Nevertheless, we also discuss an approach for treating larger displacements in Section~\ref{sect_multi} below.

As mentioned before, we want to utilize additional speckle (bubble motion) information in the displacement field reconstruction. In Section~\ref{sect_bubble}, we proposed a way to track the displacement of bubbles between two images, and thus we now assume that a set of displacement vectors $\{\ufhi =(\hat{u}^i_1, \hat{u}^i_2)\}_{i=1}^M$ with $\ufhi \subset \R^2$ is given. Since we want our estimated displacement field to coincide with these vectors at the points $\xhi =(\hat{x}^i_1,\hat{x}^i_2) \subset \Omega$ where the centers of mass of the bubbles are located, i.e., $\uf(\xhi) \approx \ufhi$, we propose the following functional:
	\begin{equation} \label{def_S_u}
		\S(\uf) := \sum\limits_{i=1}^{M}\intVx{g(\x,\hat{\x}^i) \abs{\uf(\x)-\hat{\uf}^i}^2} \,, 
	\end{equation}
where $g(\x,\hat{\x}^i)$ denotes the Gaussian function centered around $\xhi$ with standard deviation $\sigma$, i.e.,
	\begin{equation}\label{def_g}
		g(\x,\hat{\x}^i) = \frac{1}{2\pi \sigma^2} e^{-\frac{(x_1-\hat{x}_1^i)^2 + (x_2-\hat{x}_2^i)^2}{2\sigma^2}} \,.
	\end{equation}
Furthermore, since it also makes sense to assume in addition that the displacement field $\uf$ is smooth, we propose to estimate it by minimizing the following functional:
	\begin{equation}\label{def_F_u}
	\begin{split}
		F(\uf) := 
		\intVx{\left(\nabla I \cdot \uf + I_t\right)^2} +  \alpha \RE(\uf) + \beta \S(\uf) 
		= J(\uf) + \beta \S(\uf)\,,
	\end{split}	
	\end{equation}
where $\beta \geq 0$ is another regularization parameter. Depending on the choices of $\alpha$ and $\beta$, an emphasis can be placed on either the desired smoothness of the displacement field, or on its fit to the given bubble displacement $\ufhi$ (indirectly also influenced by $\sigma$).  

We are now going to analyse the problem of minimizing the functional $F$. For that, we make use of the analysis of Schn\"orr \cite{Schn91a}, which is reviewed below, who already considered the  problem of minimizing the functional $J$ from \eqref{def_J_u}. His key idea was to rewrite $J$ in the form
	\begin{equation}\label{J_at_bt_ct}
		J(\uf) = \frac{1}{2}\at(\uf,\uf) - \bt(\uf) + \ct \,,
	\end{equation}
with the bilinear form $\at(\cdot,\cdot)$, the linear form $\bt(\cdot)$ and the constant term $\ct$ given by
	\begin{equation}\label{def_at_bt_ct}
	\begin{split}
		\at(\uf,\vf) &:= 2 \intVx{\left( \nabla I \cdot \uf \right)\left( \nabla I \cdot \vf \right)} + 2 \alpha \intVx{\nabla \uf:\nabla \vf} \,,
		\\
		\bt(\vf) &:= -2 \intVx{I_t\left(\nabla I \cdot \vf \right)} \,, 
		\qquad
		\ct := \intVx{I_t^2} \,.
	\end{split}
	\end{equation}
After showing that $\at(\cdot,\cdot)$ is symmetric and elliptic on a suitable function space $V$, Schn\"orr derived results about the minimization of $J$ by noting that its unique minimizer is given as the unique solution of the linear equation
	\begin{equation}
		\at(\uf,\vf) = \bt(\vf) \,, \qquad \forall \, \vf \in V \,.
	\end{equation}
The assumptions required for his analysis are collected in the following
\begin{assumption} \label{ass_Schnoerr}
Let $\Omega \subset \R^2$ be a nonempty, bounded, open, and connected set with a Lipschitz continuous boundary $\bO$ and let $\grad I \in \LiO$ and $I_t \in \LtO$. Furthermore, let $\alpha > 0$ and let the components of $\grad I$ be linearly independent.
\end{assumption}

Under the above assumptions, Schn\"orr derived the results summarized in 
\begin{theorem}\label{thm_Schnoerr}
Let Assumption~\ref{ass_Schnoerr} hold and let the bilinear form $\at(\cdot,\cdot)$, the linear form $\bt(\cdot)$ and the constant term $\ct$ be defined as in \eqref{def_at_bt_ct}. Then $\at(\cdot,\cdot)$ and $\bt(\cdot)$ are bounded on $\HoO^2$ and $\at(\cdot,\cdot)$ is elliptic on $\HoO^2$, i.e., there exists a constant $C_E > 0$ such that
	\begin{equation}\label{at_elliptic}
		\at(\vf,\vf) \geq C_E \norm{\vf}_\HoO^2 \,,
		\qquad
		\forall \, \vf \in V  \,.
	\end{equation}
Furthermore, the unique minimizer of the problem
	\begin{equation}\label{min_J}
		\min\limits_{\uf \in \HoO^2} J(\uf) \,,
	\end{equation}
is given as the unique solution $\uf \in \HoO^2$ of the linear problem
	\begin{equation*}\label{at_b}
		\at(\uf,\vf) = \bt(\vf) \,, \qquad \forall \, \vf \in \HoO^2 \,.
	\end{equation*}
Furthermore, the solution of this equation depends continuously on the right-hand side $\bt(\cdot)$, but not necessarily on the image intensity function $I$. 
\end{theorem}
\begin{proof}
The proofs of these statements can all be found in \cite{Schn91a}.
\end{proof}

We now apply the same ideas and results presented above to the problem of minimizing the functional $F$. Note that Schnörr's work would correspond to the case $\beta=0$. First of all, in analogy to \eqref{def_at_bt_ct}, we define the bilinear form $a(\cdot,\cdot)$, the linear form $b(\cdot)$, and the constant term $c$ by
	\begin{equation}\label{def_a_b_c}
	\begin{split}
		a(\uf,\vf) &:= \at(\uf,\vf) +   2 \beta \sum\limits_{i=1}^{M} \intVx{g(\x,\hat{\x}^i) \left(\uf \cdot \vf\right)} \,,
		\\
		b(\vf) &:=  \bt(\vf) + 2 \beta \sum\limits_{i=1}^{M} \intVx{g(\x,\hat{\x}^i) \left(\hat{\uf}^i \cdot \vf\right)} \,, 
		\\
		c &:= \ct + \beta \sum\limits_{i=1}^{M} \intVx{g(\x,\hat{\x}^i) \abs{\hat{\uf}^i}^2} \,,
	\end{split}
	\end{equation}
with $\at(\cdot,\cdot)$, $\bt(\cdot)$, and $\ct$ as defined in \eqref{def_at_bt_ct}. It is easy to see that in complete analogy to the representation \eqref{J_at_bt_ct} of $J$ there holds
	\begin{equation}\label{F_a_b_c}
		F(\uf) = \frac{1}{2} a(\uf,\uf) - b(\uf) + c \,. 
	\end{equation}
We now proceed by proving some important results for the bilinear and linear form in
\begin{proposition} \label{prop_a_b}
Let Assumption~\ref{ass_Schnoerr} hold and let the bilinear form $a(\cdot,\cdot)$ and the linear form $b(\cdot)$ be defined as in \eqref{def_a_b_c}. Then $a(\cdot,\cdot)$ is bounded and elliptic on $\HoO^2$ and $b(\cdot)$ is bounded on $\HoO^2$.
\end{proposition}
\begin{proof}
We start by showing that the linear form $b(\cdot)$ is bounded. By the definition of $b$ and since, by Theorem~\ref{thm_Schnoerr}, $\bt$ is bounded, it remains to show that
	\begin{equation*}
	\begin{split}
		\abs{2 \beta\sum\limits_{i=1}^{M} \intVx{g(\x,\hat{\x}^i) \left(\hat{\uf}^i \cdot \vf\right)}} \leq C \norm{\vf}_\HoO \,.
	\end{split}
	\end{equation*}
Using the Cauchy-Schwarz inequality, we get that
	\begin{equation*}
	\begin{split}
		&\abs{\sum\limits_{i=1}^{M} \intVx{g(\x,\hat{\x}^i) \left(\hat{\uf}^i \cdot \vf\right)}}
		\leq
		\sum\limits_{i=1}^{M} \norm{g(\cdot,\hat{\x}^i)}_\infty \intVx{ \abs{\hat{\uf}^i \cdot \vf}}
		\\
		& \qquad
		\leq
		\sum\limits_{i=1}^{M} \sqrt{\intVx{ \abs{\hat{\uf}^i}^2}
		\intVx{ \abs{\vf}^2}	}
		=
		\norm{\vf}_\LtO \kl{\sqrt{ \abs{\Omega} }
		\sum\limits_{i=1}^{M}  \abs{\hat{\uf}^i}} \,,
	\end{split}
	\end{equation*}
where we have used that as a Gaussian function, $\abs{g(\x,\xhi)} \leq 1$. From this, it now follows that $b(\cdot)$ is bounded on $\HoO^2$. Similarly, for the bilinear form $a(\cdot,\cdot)$, we only need to show that
	\begin{equation*}
	\begin{split}
		\abs{2 \beta \sum\limits_{i=1}^{M} \intVx{g(\x,\hat{\x}^i) \left(\uf \cdot \vf\right)}}
		\leq 
		C \norm{\uf}_\HoO \norm{\vf}_\HoO \,.
	\end{split}
	\end{equation*}
Again using the Cauchy-Schwarz inequality, we get that
	\begin{equation*}
	\begin{split}
		\abs{\sum\limits_{i=1}^{M} \intVx{g(\x,\hat{\x}^i) \left(\uf \cdot \vf\right)}}
		\leq
		\sum\limits_{i=1}^{M} \norm{g(\cdot,\hat{\x}^i)}_\infty \intVx{ \abs{\uf \cdot \vf}}
		\leq
		M \norm{\uf}_\LtO \norm{\vf}_\LtO \,,
	\end{split}
	\end{equation*}
which implies the boundedness of $a(\cdot,\cdot)$. Thus, it remains to show that $a(\cdot,\cdot)$ is coercive. For this, we look at
	\begin{equation*}
		a(\vf,\vf) = \at(\vf,\vf) +   2 \beta \sum\limits_{i=1}^{M} \intVx{g(\x,\hat{\x}^i) \left(\vf \cdot \vf\right)} 
		\geq \at(\vf,\vf) \,,
	\end{equation*}
where we used that $g(\x,\xhi) > 0$. Hence, the coercivity of $a(\cdot,\cdot)$ now follows from the coercivity of $\at(\cdot,\cdot)$ shown in Theorem~\ref{thm_Schnoerr}, which concludes the proof.
\end{proof}

With this, we are now able to prove well-posedness of our optical flow algorithm.

\begin{theorem}\label{thm_main}
Let Assumption~\ref{ass_Schnoerr} hold and let the bilinear form $a(\cdot,\cdot)$ and the linear form $b(\cdot)$ be defined as in \eqref{def_a_b_c}. Then, the unique minimizer of the problem
	\begin{equation}\label{min_F}
		\min\limits_{\uf \in \HoO^2} F(\uf) \,,
	\end{equation}
is given as the unique solution $\uf \in \HoO^2$ of the linear problem
	\begin{equation}\label{a_b}
		a(\uf,\vf) = b(\vf) \,, \qquad \forall \, \vf \in \HoO^2 \,.
	\end{equation}
Furthermore, the solution of this equation depends continuously on the right-hand side $b(\cdot)$, but not necessarily on the image intensity function $I$.
\end{theorem}
\begin{proof}
This follows from the Lax-Milgram-Lemma in the same way as in Schn\"orr \cite{Schn91a}, using the representation \eqref{def_F_u} together with Proposition~\ref{prop_a_b}.
\end{proof}

For the choice $\alpha = 0$ and $\beta > 0$, i.e., when one only regularizes with the functional $\S$ but not with $\RE$, the requirement in the above theorem that the components of $\grad I$ have to be linearly independent can be dropped. The only difference is that one then has to replace the space $\HoO$ by $\LtO$ everywhere. This is summarized in

\begin{theorem}\label{thm_main_II}
Let $\Omega \subset \R^2$ be a nonempty, bounded, open, and connected set with a Lipschitz continuous boundary $\bO$ and let $\grad I \in \LiO$ and $I_t \in \LtO$. Furthermore, let $\alpha = 0$ and $\beta > 0$, and let the bilinear form $a(\cdot,\cdot)$ and the linear form $b(\cdot)$ be defined as in \eqref{def_a_b_c}. Then, the unique minimizer of the problem
	\begin{equation*}\label{min_F_L2}
		\min\limits_{\uf \in \LtO^2} F(\uf) \,,
	\end{equation*}
is given as the unique solution $\uf \in \LtO^2$ of the linear problem
	\begin{equation*}\label{a_b_L2}
		a(\uf,\vf) = b(\vf) \,, \qquad \forall \, \vf \in \LtO^2 \,.
	\end{equation*}
Furthermore, the solution of this equation depends continuously on the right-hand side $b(\cdot)$, but not necessarily on the image intensity function $I$.
\end{theorem}
\begin{proof}
Since the domain $\Omega$ is bounded, it follows from the definition \eqref{def_g} of the functions $g(\x,\hat{\x}^i)$ that there exists a constant $c_g > 0$ such that
	\begin{equation*}
		g(\x,\hat{\x}^i) \geq c_g \,, 
		\qquad 
		\forall \, x \in \Omega \,.
	\end{equation*}
Since this implies that
	\begin{equation*}
		2 \sum\limits_{i=1}^{M} \intVx{g(\x,\hat{\x}^i) \left(\vf \cdot \vf\right)} 
		\geq 
		2 M c_g \norm{v}_\LtO^2  \,,
	\end{equation*} 
it follows that $a(\cdot,\cdot)$ is elliptic on $\LtO^2$ for $\alpha = 0$ and $\beta > 0$. Furthermore, it is easy to see that in this case $a(\cdot,\cdot)$ and $b(\cdot)$ are also bounded on $\LtO^2$. Hence, the statements of the theorem now follow by the Lax-Milgram-Lemma using the representation \eqref{def_F_u} in the same way as in \cite{Schn91a} or Theorem~\ref{thm_main_II} above.
\end{proof}

\begin{remark}
An extension of the above approach is possible if one deals with (almost) incompressible materials. Since for incompressible materials there holds $\div{\uf} = 0$, it is then natural to consider the additional regularization functional
	\begin{equation*}
	\begin{split}
		\mathcal{T}(\uf) &:= \norm{\div{\uf}}^2_{\LtO} \,,
	\end{split}
	\end{equation*}
and instead of \eqref{min_F} to solve the adapted minimization problem
	\begin{equation*}
	\begin{split}
		\min\limits_{\uf \in \HoO^2} \Kl{ F(\uf) + \gamma \, \mathcal{T}(\uf) }\,,
	\end{split}	
	\end{equation*}
where $\gamma > 0$ is an additional regularization parameter. The theoretical results derived above still remain valid, and thus one can expect to obtain a physically meaningful displacement field reconstruction also in this case. However, one now has to tune three regularization parameters instead of two. Note that for incompressible materials also the subsequent material parameter estimation step (see Section~\ref{sect_application} below) changes, since for example under models of linear elasticity one then no longer searches for the Lam\'e parameters $\lambda$ and $\mu$, but instead only for $\mu$ and the pressure \cite{LiWijDanSamMunKenObe2016, LiWijSamKenMunObe2019}. 
\end{remark}

\section{Minimization Approaches}\label{sect_mini}

In this section, we are discussing two alternative approaches, direct and iterative, for the minimization of the functional \eqref{def_F_u}. We analyze them based on the theoretical results presented above.

\subsection{Direct Methods}\label{sect_direct}

Due to Theorem~\ref{thm_main}, one way of computing the minimizer of the functional $F$ is by solving the variational problem \eqref{a_b}, which can for example be done via a \emph{finite element (FE) method}. Let $\{\psi_1\,,\dots\,, \psi_N\}$ be a basis of a finite dimensional subspace $V_h$ of $\HoO^2$, consisting for example of piecewise linear functions on a suitable triangulation of the domain $\Omega$. The FE approximation $u_h$ to the solution of \eqref{a_b} then is
	\begin{equation*}
		u_h = \sum\limits_{k=1}^N z_k \psi_k \,,
	\end{equation*}
where the vector $z = (z_k)$ is given as the solution of the matrix-vector equation $A z = y$, with the $N \times N$ matrix $A = (A_{ik})$ and the $N \times 1$ vector $y = (y_i)$ defined by
	\begin{equation*}
		A_{ik} := a(\psi_k,\psi_i) \,,
		\qquad
		\text{and}
		\qquad
		y_i := b(\psi_i) \,.
	\end{equation*}
Due to the  ellipticity of the bilinear form $a(\cdot,\cdot)$, the matrix $A$ is positive definite, and thus the matrix-vector system $Az=y$ admits a stable solution. Of course, many different versions of finite element methods are applicable to the variational problem, each with its own peculiarities (see for example \cite{Bra07,McL00,Nec11}). Also, for the solution of the resulting matrix-vector equation a number of different solvers are possible, from simple direct solvers to preconditioned iterative solvers. In particular, due to the underlying structure of the problem, multi-level (multi-grid) ideas are preferable, which not only lend themselves to parallelization, but can also be adapted for improving the estimation of relatively large displacement fields (see Section~\ref{sect_multi} on multi-scale approaches). 

\subsection{Iterative Methods}

In some cases, the minimization of $F$ via directly solving the underlying variational problem \eqref{a_b} as discussed in the previous section is either undesirable or infeasible, for example if the size of the underlying images are too large. In this case, one can apply iterative methods for the minimization of $F$, which we discuss in this section.

First of all, due to the representation of $F$ given in \eqref{F_a_b_c}, we can derive the continuity and differentiability results presented in the following two propositions.
\begin{proposition}
Let Assumption~\ref{ass_Schnoerr} hold. Then for the functional $F$ defined in \eqref{def_F_u} and for all $\uf,\ufb \in \HoO^2$ there holds
	\begin{equation*}
		\abs{F(\uf) - F(\bar{\uf})}	\le C\norm{\uf - \bar{\uf}}_\HoO \left( \norm{\uf-\bar{\uf}}_\HoO +  \norm{\bar{\uf}}_\HoO  + 1 \right)\,,
	\end{equation*}
for a constant $C > 0$, and thus $F$ is continuous on $\HoO^2$.
\end{proposition}
\begin{proof}
It follows from the representation \eqref{F_a_b_c} together with Proposition~\ref{prop_a_b} that
	\begin{equation*}
	\begin{split}
		& \abs{F(\uf) - F(\bar{\uf})} 
		= \abs{\frac{1}{2} a(\uf,\uf) - b(\uf) - \frac{1}{2} a(\bar{\uf},\bar{\uf}) + b(\bar{\uf})} 
		\\
		&
		\qquad
		= \abs{\frac{1}{2} a(\uf-\bar{\uf},\uf-\bar{\uf}) + a(\uf-\bar{\uf},\bar{\uf}) + b(\bar{\uf}-\uf)} 
		\\
		& 
		\qquad
		\le C \kl{\norm{\uf-\bar{\uf}}^2_\HoO +  \norm{\bar{\uf}}_\HoO \norm{\uf-\bar{\uf}}_\HoO + \norm{\bar{\uf}-\uf}_\HoO} 
		\\
		& 
		\qquad
		=
		C \norm{\uf-\bar{\uf}}_\HoO \left( \norm{\uf-\bar{\uf}}_\HoO +  \norm{\bar{\uf}}_\HoO  + 1 \right) \,,
	\end{split}
	\end{equation*}
from which the assertion follows.
\end{proof}

\begin{proposition}\label{prop_F_derivative}
Let Assumption~\ref{ass_Schnoerr} hold. Then the functional $F$ defined in \eqref{def_F_u} is twice continuously Fr\'echet differentiable on $\HoO^2$ with
	\begin{equation}
	\begin{split}
		F'(\uf)(\hf) = a(\uf,\hf) - b(\hf)\,,
		\qquad
		\text{and}
		\qquad
		F''(\uf)(\hf,\wf) = a(\wf,\hf)
	\end{split}.
	\end{equation}	
\end{proposition}
\begin{proof}
This is a direct consequence of the representation \eqref{F_a_b_c} of $F$ in terms of the bilinear form $a(\cdot,\cdot)$, the linear form $b(\cdot)$, and the constant $c$, together with the boundedness results from Proposition~\ref{prop_a_b}.
\end{proof}

As a consequence of the above result, we get
\begin{corollary}\label{cor_convexity}
Let Assumption~\ref{ass_Schnoerr} hold. Then the functional $F$ defined in \eqref{def_F_u} is strictly convex on $\HoO^2$.
\end{corollary}
\begin{proof}
Due to Proposition~\ref{prop_F_derivative} and Proposition~\ref{prop_a_b} there holds
	\begin{equation*}
		F''(\uf)(\hf,\hf) = a(\hf,\hf) \geq C_E \norm{\hf}_\HoO^2 \,,
	\end{equation*} 
and thus it now follows from \cite[Proposition~17.3]{BauCom11} that $F$ is convex.
\end{proof}

The above results imply that any first- or second-order iterative method can applied for the minimization of $F$, especially since due to the strict convexity of $F$ (see Corollary~\ref{cor_convexity}) the minimizer is unique. In particular, the first-order optimality condition
	\begin{equation*}
		F'(\uf)h = 0 \,,
		\qquad 
		\forall \, h \in \HoO^2 \,, 
	\end{equation*}
is both necessary and sufficient in this case (compare with Theorem~\ref{thm_main}). Since, by the Riesz representation theorem, there exists a unique element $\nabla F(\uf) \in \HoO^2$ (the \emph{gradient}) such that
	\begin{equation*}
		F'(\uf)h  = \spr{\nabla F(\uf),h} \,,
		\qquad \forall \, h \in \HoO^2 \,,
	\end{equation*}
the first-order optimality condition is equivalent to
	\begin{equation*}
		\nabla F(\uf) = 0 \,.
	\end{equation*}
	
Applying for example the idea of fixed-point iteration to this equation yields the simple and well-known \emph{gradient descent} method 
	\begin{equation*}
		\uf_{k+1} = \uf_k - \omega_k F'(\uf_k)\,,
	\end{equation*}
where $\omega_k$ is a suitably chosen stepsize. For our particular problem, it follows from the explicit expression of $F'(\uf)h$ that the gradient $F'(\uf) \in \HoO^2$ is given as the unique solution of the variational problem
	\begin{equation*}
		\spr{F'(\uf), \vf}_{\HoO} = a(\uf,\vf) - b(\vf)\,, \qquad \forall \, \vf \in \HoO^2 \,,
	\end{equation*}
and can therefore be easily computed numerically, using for example again a finite-element method, which in this case can be implemented very efficiently, since the bilinear form is now just the scalar product of $\HoO^2$. 

Of course there is a multitude of different iterative methods which can be applied to our minimization problem. The above considerations were mainly made to underline that since we showed that the problem is well-behaved in many ways, most of them are applicable and should lead to good results.

\subsection{Multi-Scale Approach}\label{sect_multi}

In this section, we consider the so-called \emph{multi-scale approach} (see for example \cite{LauKorMem04,MeiSanKon13,Mod09,Mod03}), for improving the reconstruction quality of the proposed methods in a numerical implementation, in particular for larger displacement fields. As the name suggests, one considers the optical flow problem on various scales, i.e., levels of discretization, by a suitable downscaling of the input images. On each of the scales, one solves the optical flow problem and then combines the results from each scale to obtain a final optical flow estimation. The idea behind this approach is that by considering the problem on multiple scales, information about the flow on those scales can be extracted more reliably than from the original input images alone. For example, due to the inherent locality of their estimation approaches, many optical flow algorithms often end up producing flow fields on fine scales only, and are not able to capture larger displacements. The multi-scale approach is a way to produce flow fields capturing larger displacements as well.

Following the ideas and notation of \cite{MeiSanKon13}, we start by fixing a number $N$ of scales, indexed by $s$ and ranging from the finest (original) scale $s = 0$ to the coarsest scale $s = N-1$. Fixing some down-sampling factor $\eta \in (0,1)$, the image intensity function on the scale $s$, denoted by $I^s(\x,t)$, is then computed recursively via the formula
	\begin{equation*}
		I^s(\eta \x,t) := G_\sigma \ast I^{s-1}(\x,t) \,,
	\end{equation*}
where $G_\sigma$ is a Gaussian filter kernel with standard deviation $\sigma$ and $I^0(\x,t)$ is the image intensity function of the original image. The value of $\sigma$ depends on $\eta$ and is chosen as
	\begin{equation}\label{def_sigma_eta}
		\sigma(\eta) := \sigma_0 \sqrt{\eta^{-2}-1} \,,
	\end{equation}
where the authors of \cite{MeiSanKon13} suggest $\sigma_0 := 0.6$ from experience with numerical examples. The multi-scale approach then consists in solving the optical flow problem on the scale $s$ from the images $I^s$ and to use the obtained displacement field $\uf^s$ as an initial information for the computation on the scale $s-1$. Besides a method for solving the optical flow problem on any scale $s$, all that is required in addition for the application of the multi-scale approach is an up-scaling routine to transfer the obtained solution $\uf^s$ to the finer scale $s-1$. This can for example be achieved by a suitable interpolation routine.

For the iterative solution approaches presented above, incorporating the multi-scale idea is straightforward. Starting from the coarsest scale, one applies the chosen iterative method until a suitable approximation of the displacement field is obtained. This field is then transferred to a finer grid and serves as the initial guess for the iteration on that scale, and so on.

For the direct methods discussed above, the multi-scale approach combines well with a multi-grid finite element approach \cite{Bra07}. For this, we consider a sequence $(V_h^s)_{s=0}^{N-1}$ of nested finite element spaces and denote by $a^s(\cdot,\cdot)$ and $b^s(\cdot)$ the bilinear and linear forms $a(\cdot,\cdot)$ and $b(\cdot)$, in which $\nabla I$ and $I_t$ are replaced by their Gaussian-filtered counterparts $\grad I^s$ and $I^s_t$, respectively. On the coarsest level $s=N-1$, one then computes the flow field $\uf^{N-1} \in V_h^{N-1}$ as the solution of the variational problem
	\begin{equation*}
		a^{N-1}(\uf^{N-1},\vf) = b^{N-1}(\vf) \,, \qquad \forall \, \vf \in V_h^{N-1} \,.
	\end{equation*} 
For all other scales $0 \leq s \leq N - 2$, the flow field $\uf^{s} \in V_h^{s}$ is computed recursively via
	\begin{equation*}
		\uf^{s} = P_{s-1}^{s} \uf^{s-1} +  \hf^{s} \,,
	\end{equation*}
where $P_{s-1}^{s}$ denotes the orthogonal projector from $V_h^{s-1}$ to $V_h^{s}$ and $\hf^{s} \in V_h^{s}$ is defined as the (unique) solution of the variational problem
	\begin{equation*}
		a^{s}(\hf^{s},\vf) = b^{s}(\vf) - a^{s}(P_{s-1}^{s}\uf^{s-1},\vf) \,,
		\qquad
		\forall \, \vf \in V_h^{s} \,.
	\end{equation*}
The flow field $\uf^0$ computed on the finest scale $V_h^0$ is then chosen as the approximation to the solution of the displacement field estimation problem.

\section{Numerical Results I: Displacement Estimation}\label{sect_numerics}

In this section, we present a number of numerical examples based on both simulated and experimental data demonstrating the usefulness of our proposed method for displacement field estimation. Numerical examples of material parameter reconstruction based on the obtained displacement fields are presented in Section~\ref{sect_numerics_II} below. The presented results were obtained using the direct method introduced in Section~\ref{sect_direct} together with the multi-scale approach from Section~\ref{sect_multi}. We used conforming, piecewise linear finite elements on a regular triangulation of the square (image) domain $\Omega$.  The implementation was done in Python using the finite-element library Fenics \cite{AlnBleHakJohKeh15} on a desktop computer running on a Linux OS with an 8 core processor (Intel(R) core(TM) i7-7700 CPU@3.60GHz) and 62.9GB RAM.

All our numerical results were obtained from two images $I_1$ and $I_2$, as would be available from an OCT scan before and after compression of a sample. This means that the image intensity function $I(\x,t)$ is only given at two time points. Hence, for the numerical implementation we approximated the temporal derivative $I_t$ in the definition of the bilinear and linear forms $a(\cdot,\cdot)$ and $b(\cdot)$ by a backward difference quotient. The spatial gradient $\nabla I$ was approximated by $\nabla I_1$, which is computed as the spatial gradient of the finite-element function representing the image $I_1$.

\subsection{Simulated Data}

In this section, we apply our displacement field estimation approach to two problems with simulated data, the first being an optical flow test problem and the second mimicking an experimental quantitative elastography setup.

\subsubsection{Moving Squares}

The first problem, which is for example used in \cite{AubDerKor99,Schn91a} for testing optical flow estimation methods, is the problem of the moving squares. Therein, one considers two squares of constant image intensity moving towards each other. This is a difficult test problem, since the gradient of the images only contains information along the borders of the squares, and it admits a number of different plausible solutions. It is well-known that the reconstruction quality increases if for example the images are overlaid with an inhomogeneous background.

In order to test our proposed approach, we thus consider the two images $I_1$ and $I_2$ depicted in Figure~\ref{fig_squares_speckle}, which show two moving squares with added bubbles, as they would for example result from an OCT scan. Also depicted in Figure~\ref{fig_squares_speckle} is the displacement field acting between the two images, which we want to recover.

\begin{figure}[ht!]
	\centering
	\includegraphics[height = 0.16\textheight, clip=true, trim = {0.5cm 1cm 1cm 0cm}]{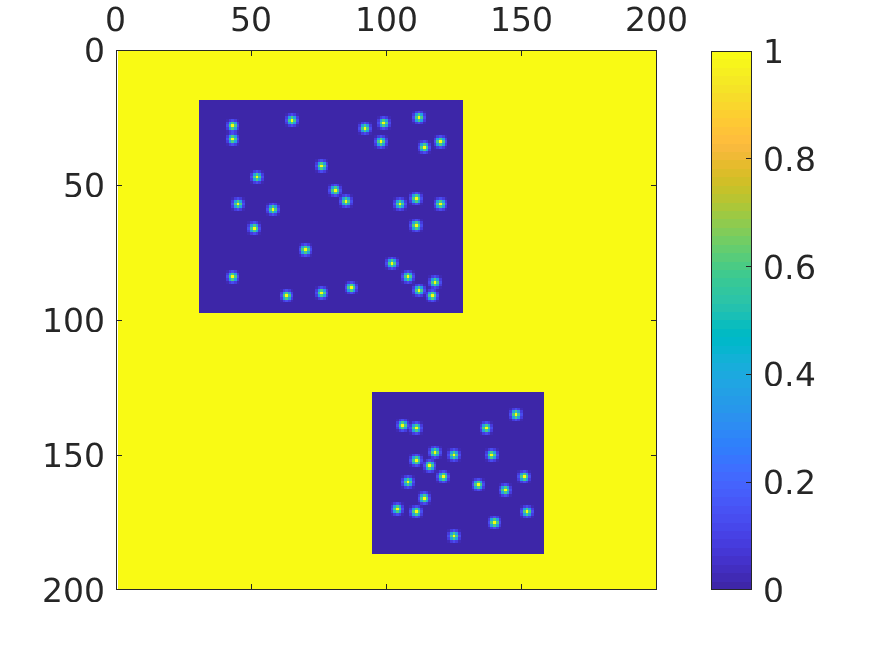}
	\quad
	\includegraphics[height = 0.16\textheight, clip=true, trim = {0.5cm 1cm 1cm 0cm}]{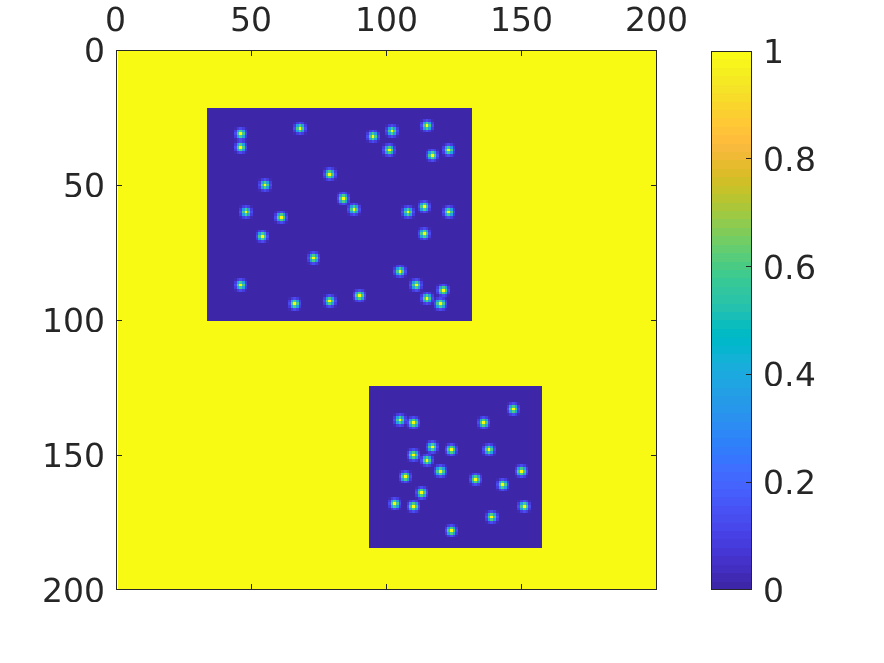}
	\quad
	\includegraphics[height = 0.16\textheight, clip=true, trim = {1cm 1cm 0cm 0.5cm}]{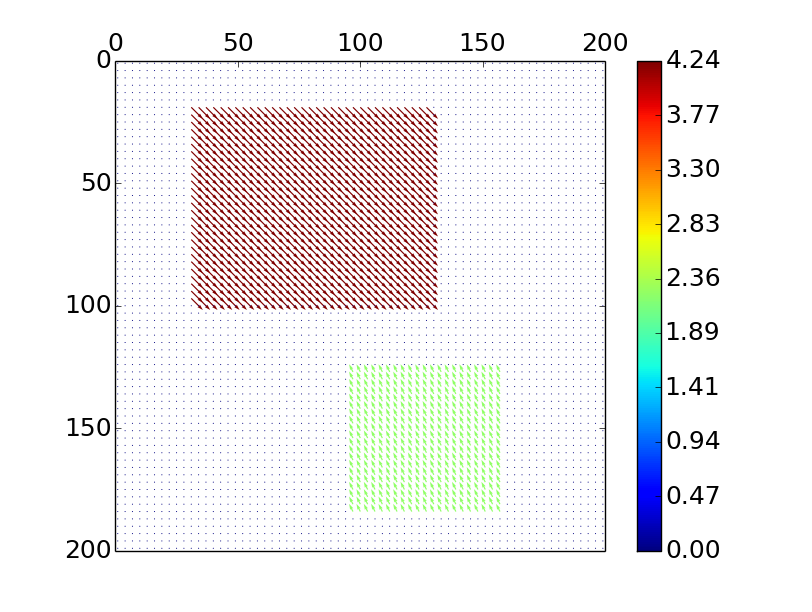}
	\caption{Test images $I_1$ (left) and $I_2$ (middle) of two moving squares with bubbles, as well as the displacement field encoding their movement, i.e.\ the ground truth (right).}
	\label{fig_squares_speckle}
\end{figure}

We want to show that utilizing the bubble motion information leads to an improved displacement field reconstruction. Hence, for the following test we assume that the displacement of all the $50$ bubbles in the test images is known. However, in order to model the unavoidable inaccuracies in the tracking process, we add $0.1 \%$ relative noise to the whole set of the displacement vectors, which is already much compared to the magnitude of a single displacement vector. 

For obtaining the results presented below, we use the following choice of parameters: For the smoothness-regularization, we choose $\alpha = 0.8$, and when utilizing the bubble information, we choose $\beta= 4$ as well as $\sigma = 5$ in the definition of $g(\x,\xhi)$ in \eqref{def_g}. For using the multi-scale approach, we take $5$ scale-levels, $\sigma_0 = 0.6$, and the grid-downscaling parameter $\eta = 0.5$, which leads to $\sigma(\eta) \approx 1.03923$ in formula \eqref{def_sigma_eta}.  

Figure~\ref{fig_squares_speckle_recon} depicts the results of the displacement field estimation, for the four different combinations of not utilizing ($\beta = 0$) or utilizing ($\beta > 0$) the bubble information, and using or not using the multi-scale approach. Both the colour and the length of the displacement vectors in the figures thereby correspond to the local magnitudes of the reconstructed displacement fields. One can clearly see that utilizing the bubble information leads to reconstructions with greatly improved quality. In particular, the general tendency of optical flow to visualize only the movement of the boundary is counteracted. Both qualitative and quantitative features of the exact displacement field (compare with Figure~\ref{fig_squares_speckle})) can be recovered accurately utilizing the bubble displacement information. Furthermore, one can also see from the plots that using the multi-scale approach helps to obtain a more uniform field that is also closer in magnitude to the exact solution.

\begin{figure}[ht!]
	\centering
	\includegraphics[width=0.45\textwidth, clip = true, trim = {1cm 1cm 0cm 0cm}]{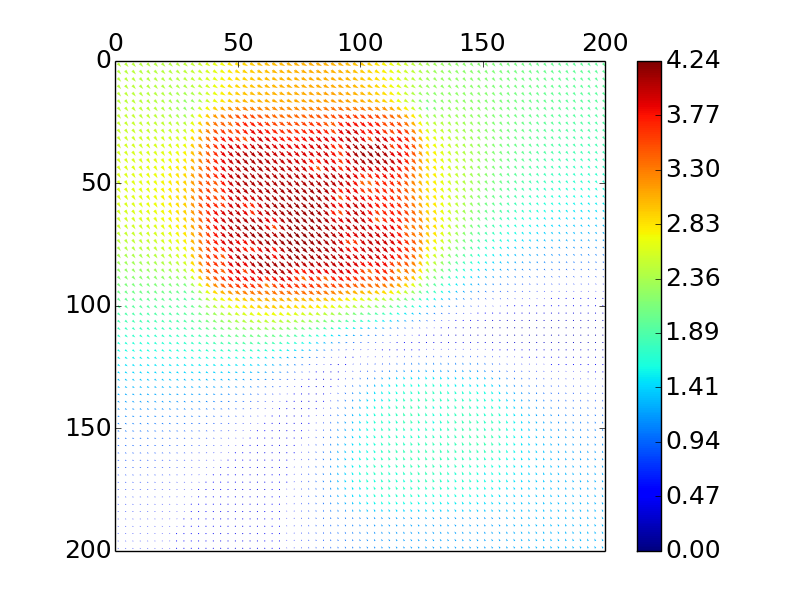}
	\includegraphics[width=0.45\textwidth, clip = true, trim = {1cm 1cm 0cm 0cm}]{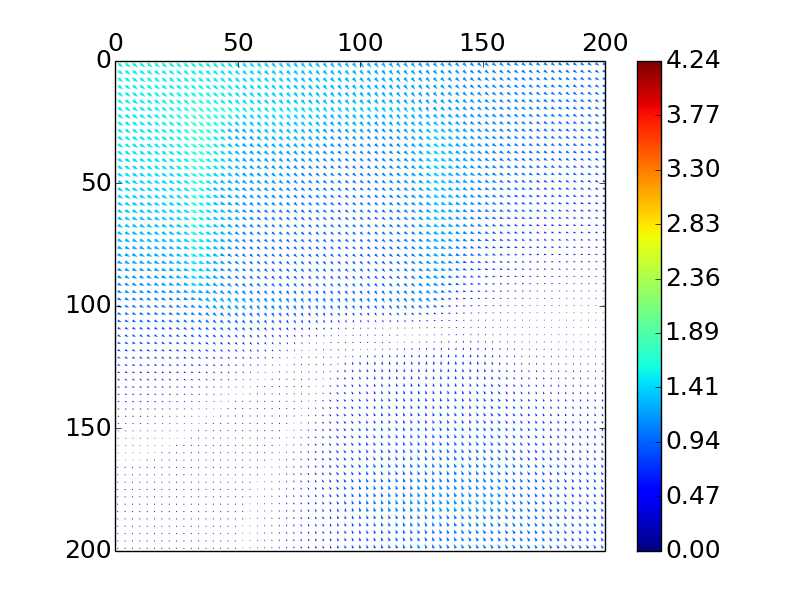}
	\\ \vspace{10pt}
	\includegraphics[width=0.45\textwidth, clip = true, trim = {1cm 1cm 0cm 0cm}]{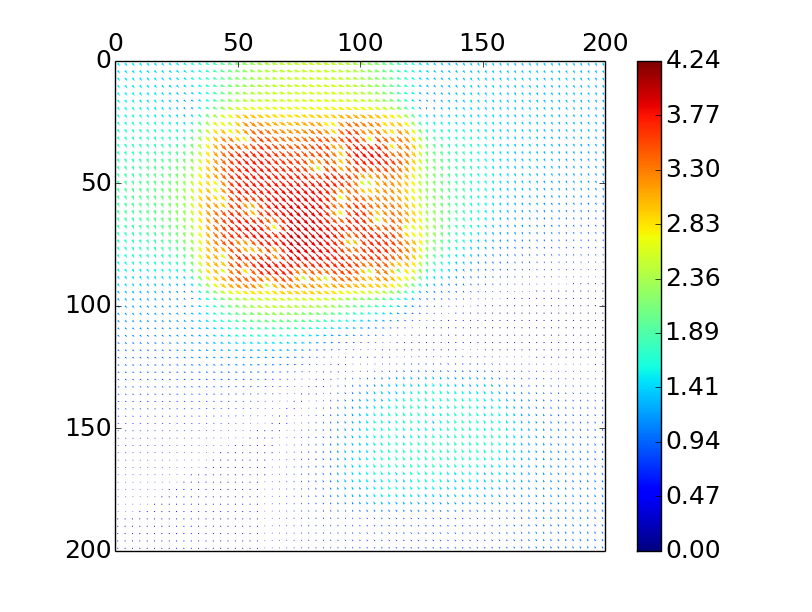}
	\includegraphics[width=0.45\textwidth, clip = true, trim = {1cm 1cm 0cm 0cm}]{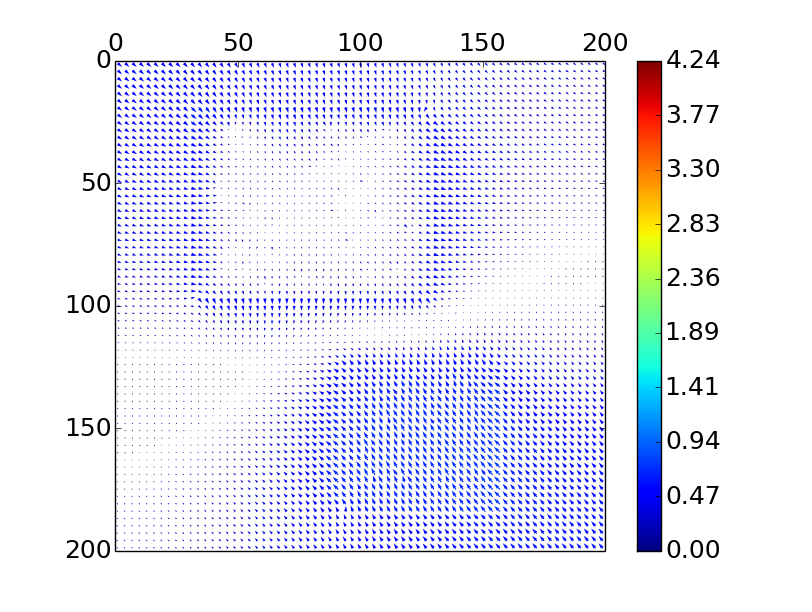}
	\caption{Estimated displacement field for the choices $\beta = 4$ (left) and $\beta = 0$ (right), both with (top) and without (bottom) the use of the multi-scale approach. Recall that the choice $\beta=0$ corresponds to the classic Horn-Schunck method while $\beta > 0$ corresponds to our proposed method. Note that for $\beta \neq 0$ the general tendency of optical flow methods to visualize only movement on the boundary is counteracted. }
	\label{fig_squares_speckle_recon}
\end{figure}

\subsubsection{Simulated Elastography Setup}\label{sect_sim_exp}

The second test problem which we now consider is designed to mimic an experimental elastography setup. We start with the rectangular sample with a circular inclusion depicted in   Figure~\ref{fig_sim_exp_sample_field} (left), which we assume to be fixed on the bottom and freely moving on the sides. We apply a constant downward displacement of $10\%$ of its size (corresponding to $20$ pixels) on the top, which results in the deformed sample depicted in Figure~\ref{fig_sim_exp_sample_field} (middle). The corresponding displacement field, depicted in Figure~\ref{fig_sim_exp_sample_field} (right), which was used to deform the sample, was computed by solving the equations of linearized elasticity \eqref{eq_elast_lin}, with the Lam\'e parameters $\lambda$ and $\mu$ as in Figure~\ref{fig_sim_exp_exact}.

\begin{figure}[ht!]
	\centering
	\includegraphics[height=0.18\textheight, clip = true, trim={1.5cm 1cm 2cm 0cm}]{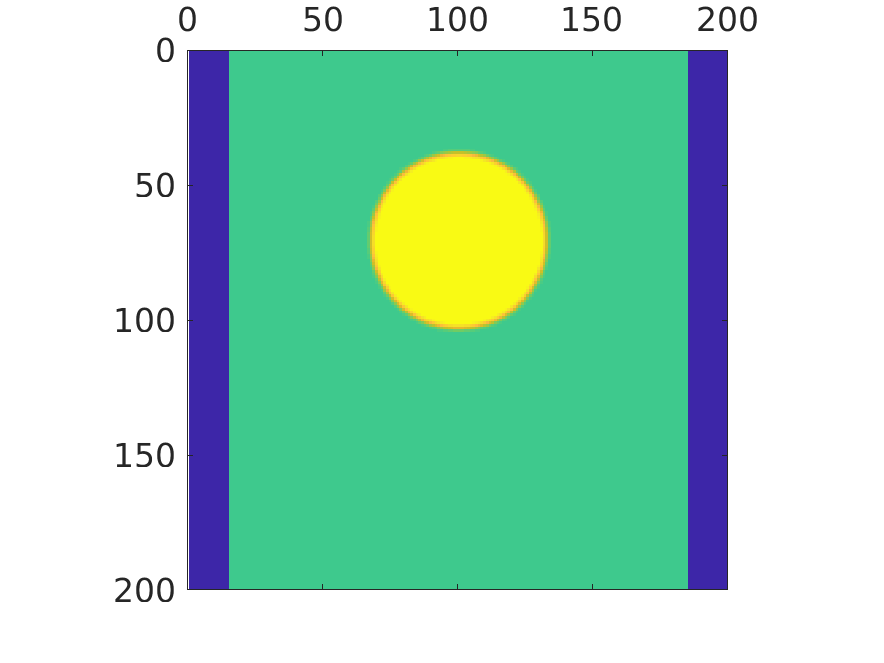}
	\quad
	\includegraphics[height=0.18\textheight, clip = true, trim={1.5cm 1cm 2cm 0cm}]{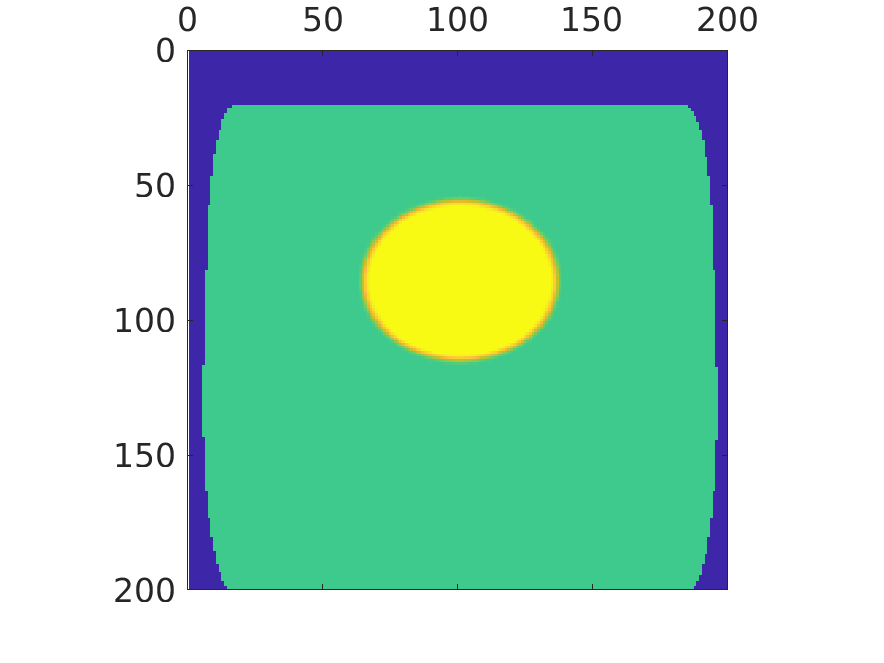}
	\quad
	\includegraphics[height=0.18\textheight, clip = true, trim={0.5cm 1cm 1cm 0cm}]{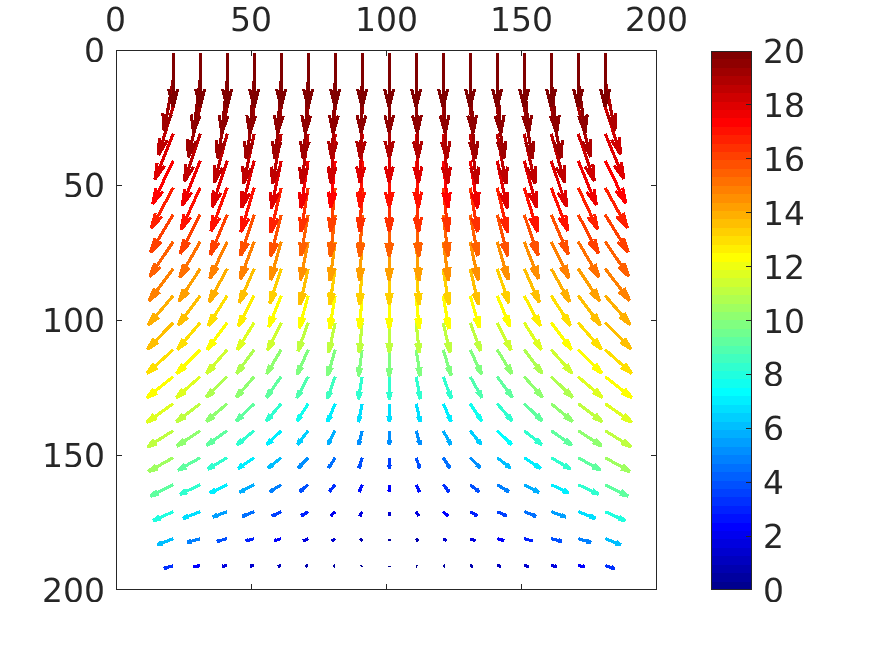}
	\caption{Simulated rectangular sample with a circular inclusion, before (left) and after (middle) compression, together with the corresponding displacement field (right).}
	\label{fig_sim_exp_sample_field}
\end{figure}

In order to simulate two OCT measurements, $200$ randomly distributed bubbles were added to the sample and the resulting images before and after compression were rescaled to the interval $[0,1]$. The movement of the bubbles between the images was also determined by the computed displacement field. The resulting simulated ``OCT images'' together with the corresponding bubble displacement vectors $\hat{\uf}^i$ are depicted in Figure~\ref{fig_sim_exp_speckle_field}.

\begin{figure}[ht!]
	\centering
	\includegraphics[height=0.16\textheight, clip = true, trim={0.5cm 1cm 1cm 0cm}]{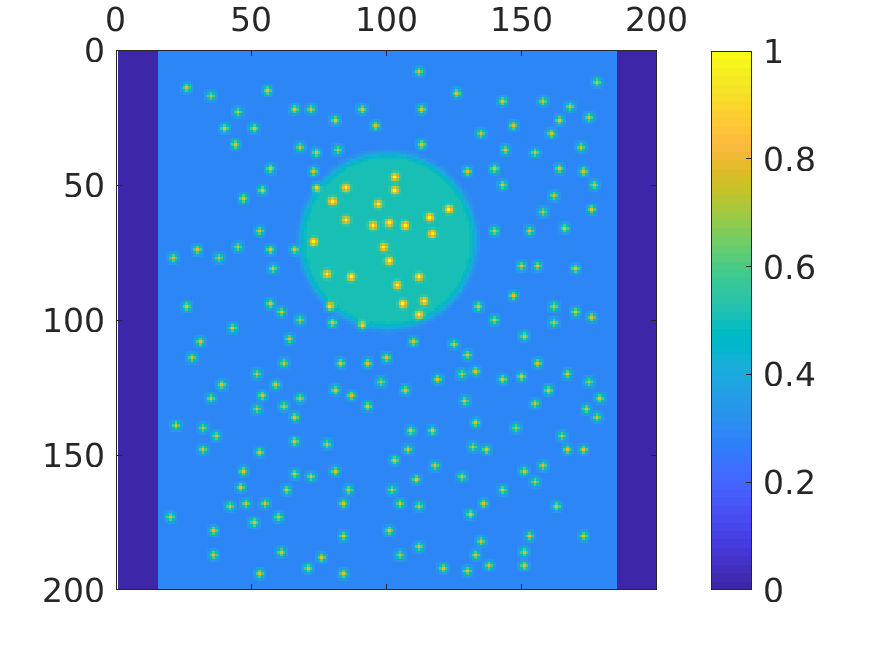}
	\quad
	\includegraphics[height=0.16\textheight, clip = true, trim={0.5cm 1cm 1cm 0cm}]{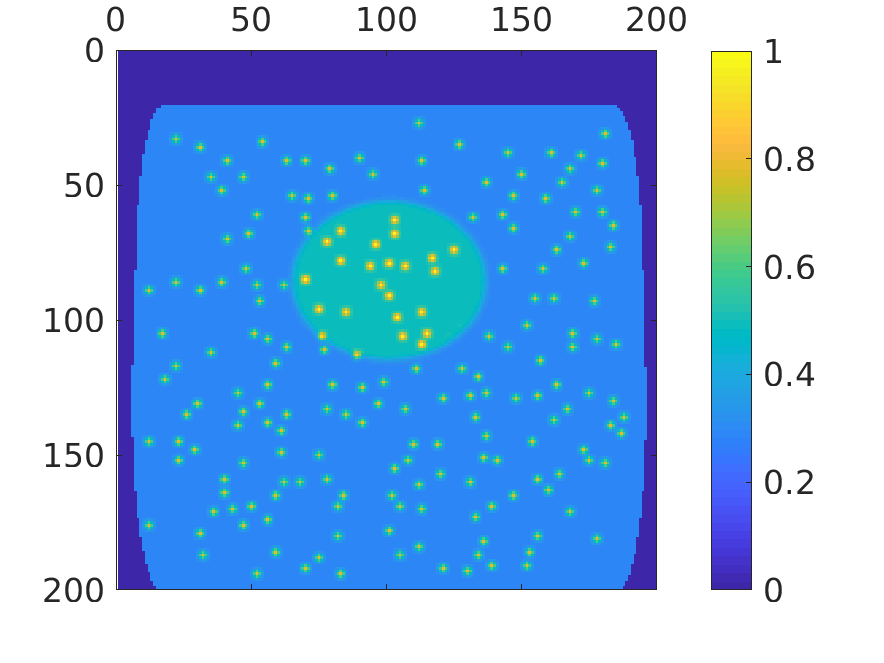}
	\quad
	\includegraphics[height=0.16\textheight, clip = true, trim={0.5cm 1cm 1cm 0cm}]{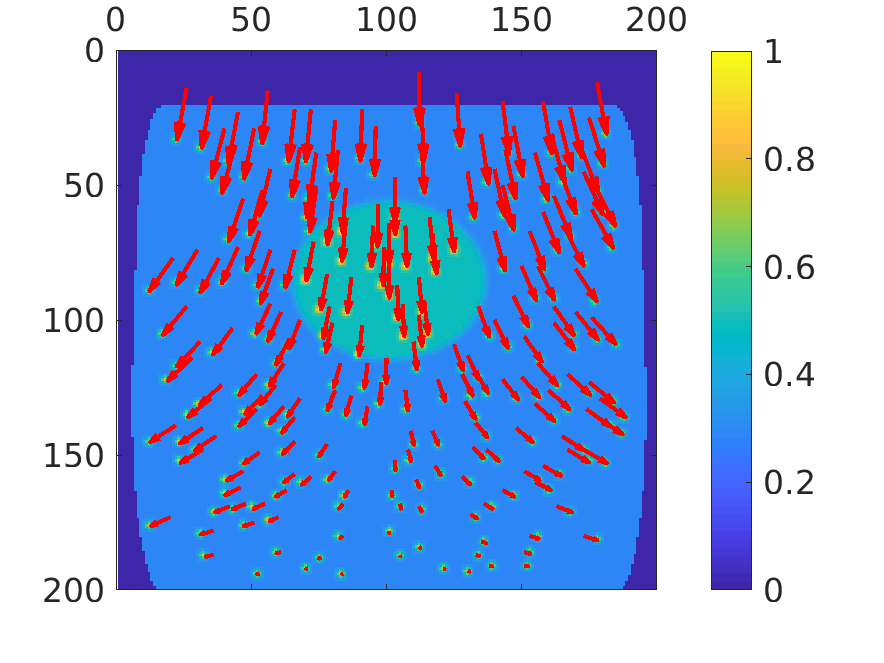}
	\caption{Simulated OCT images of the sample with added bubbles, before (left) and after (middle) compression, together with the bubble displacement vectors $\hat{\uf}^i$ (right).}
	\label{fig_sim_exp_speckle_field}
\end{figure}

In order to reconstruct the displacement field from the simulated OCT images, we employ our estimation procedure using the multi-scale approach. We chose the same parameter as for the moving squares example above (in particular $\beta = 4$), except for the spatial regularization parameter $\alpha$, for which we now use the slightly larger choice of $\alpha =4$ compared to $\alpha=0.8$ in the previous test. The reconstructed displacement field is depicted in  Figure~\ref{fig_sim_exp_fields} (left). Note that, similarly to the exact displacement field depicted in Figure~\eqref{fig_sim_exp_sample_field} (right), this is only a sparse representation of the actual field, which in reality is eight times denser in each direction. For a higher resolution image of the reconstructed displacement field see Figure~\ref{fig_sim_exp_fields} (middle). As can for example be seen from Figure~\ref{fig_sim_exp_fields} (right), the reconstructed field agrees well with the exact field, in particular in the internal part of the sample. The total relative error is $11.53\%$, with the relative error in the $x$ and $y$ components being $9.2\%$ and $6.96\%$, respectively. As the numerical examples in Section~\ref{sect_numerics_II} below show, this is sufficiently accurate to obtain quantitative material parameter reconstructions.


\begin{figure}[ht!]
	\centering
	\includegraphics[height=0.16\textheight, clip = true, trim={0.5cm 1cm 1cm 0cm}]{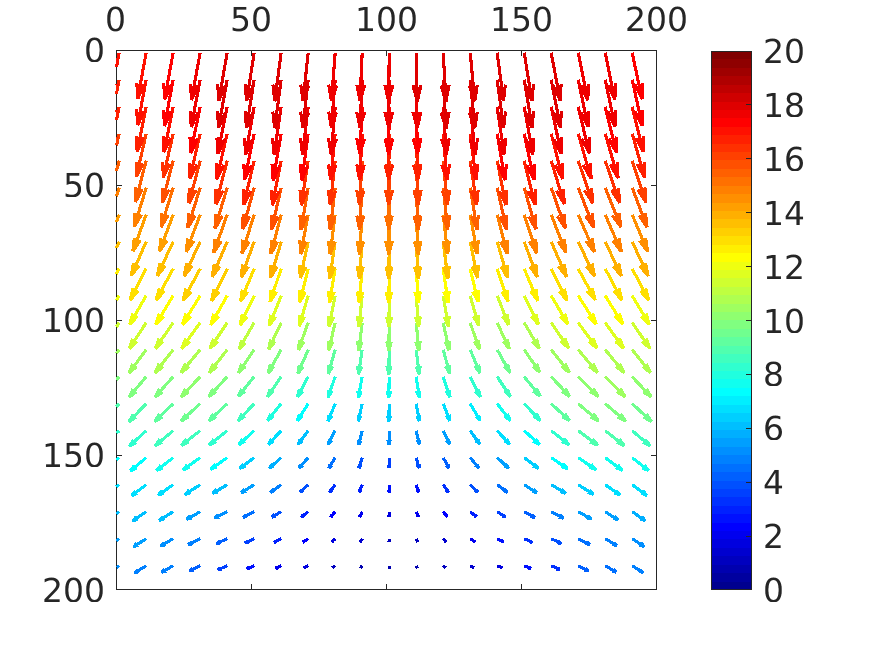}
	\quad
	\includegraphics[height=0.16\textheight, clip = true, trim={1.5cm 1cm 2.5cm 0.5cm}]{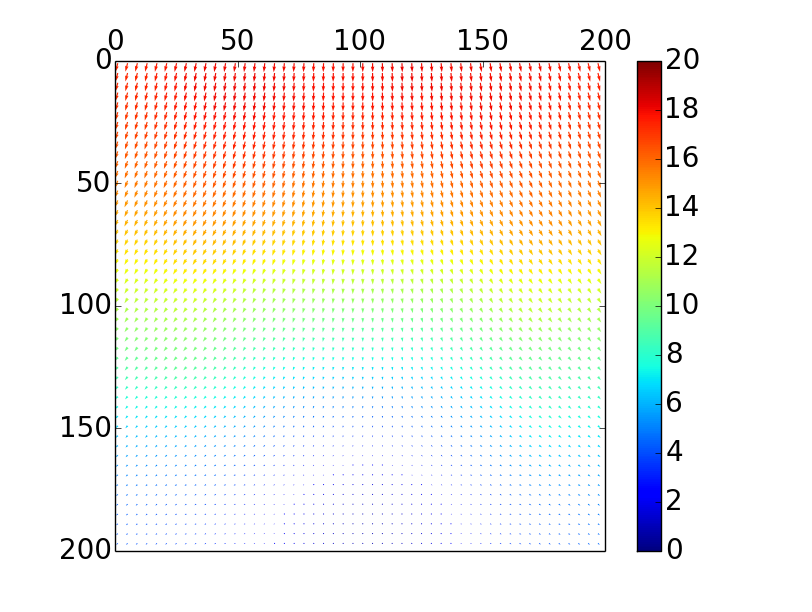}
	\quad
	\includegraphics[height=0.16\textheight, clip = true, trim={1.5cm 1cm 1cm 0cm}]{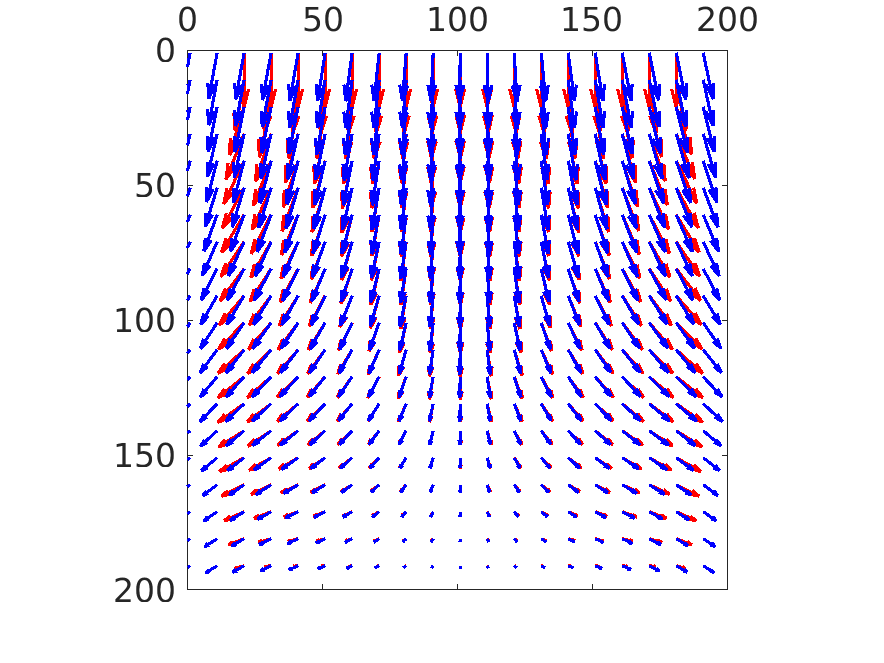}
	\caption{Estimated displacement field in low (left) and high (middle) resolution. Comparison of the exact (red) and the estimated (blue) displacement field (right).}
	\label{fig_sim_exp_fields}
\end{figure}

\subsection{Experimental Data}

After testing the proposed algorithm on simulated data, we now move to experimental data. We consider the volumetric OCT data whose rotational maximum-intensity projection (MIP) was already depicted in Figure~\ref{fig_OCT_speckle} in the introduction. This data results from the consequent OCT imaging of a homogeneous sample which was compressed from above by a micrometer screw gauge. The sample itself is made from silicone rubber and is rotationally symmetric. 

As a first test, we apply the bubble tracking algorithm presented in Section~\ref{sect_bubble} to the two OCT images shown in Figure~\ref{fig_OCT_speckle}, the results of which are depicted in Figure~\ref{fig_OCT_speckle_tracking}. In the left two images one can see the detected bubbles (circled in black) in the two images before and after compression. The threshold value in the algorithm was set so that only the bubbles of highest brightness would be detected. The right image depicts the resulting displacement vectors after the correspondence between the bubbles in the two images has been established by the algorithm. A visual inspection shows that the motion of the bubbles was reliably captured by our proposed tracking algorithm.

\begin{figure}[h!] 
	\centering
	\includegraphics[width=0.30\textwidth, clip=true, trim={12cm 2cm 10cm 0cm}]{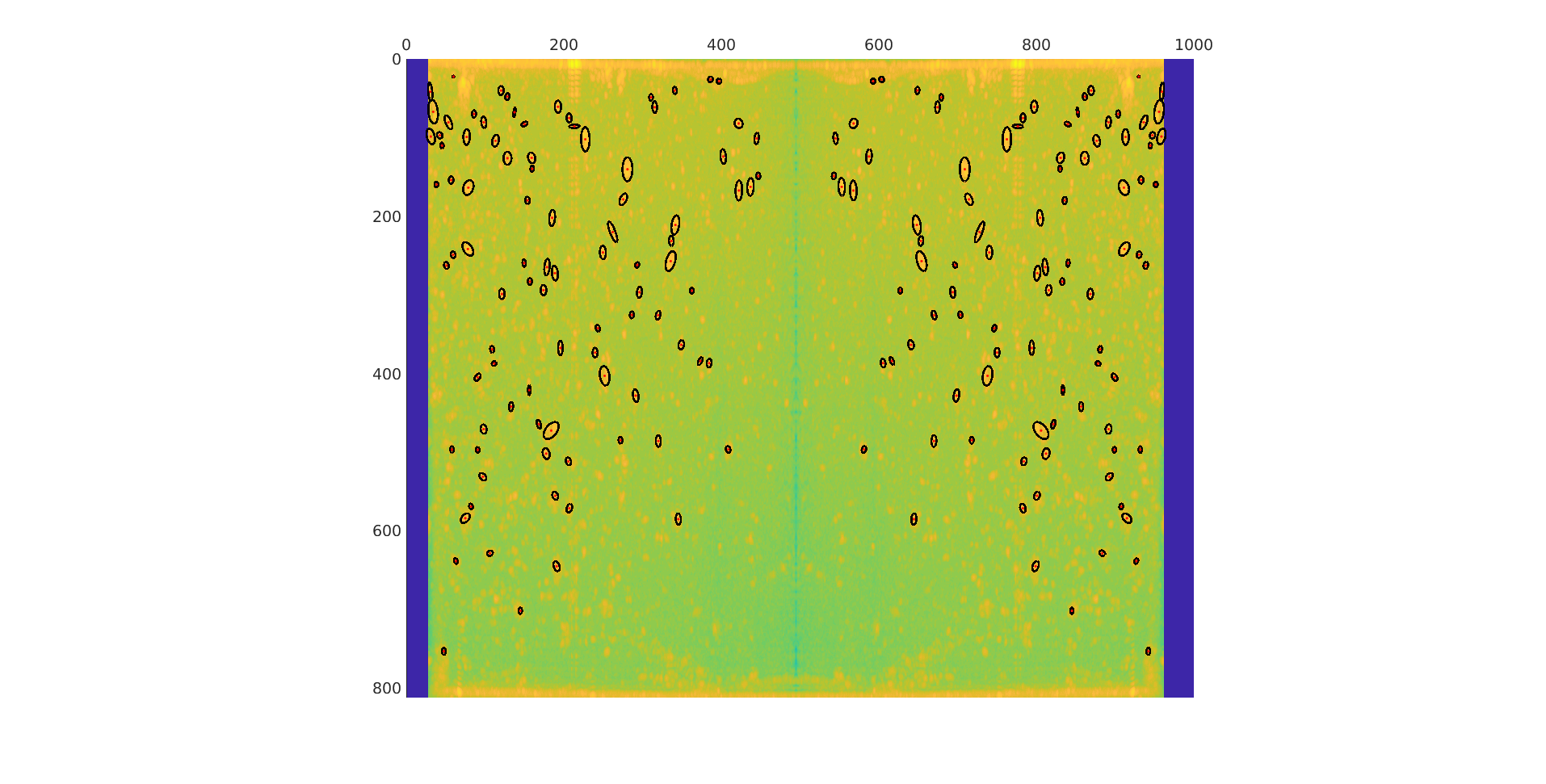}
	\quad
	\includegraphics[width=0.30\textwidth, clip=true, trim={12cm 2cm 10cm 0cm}]{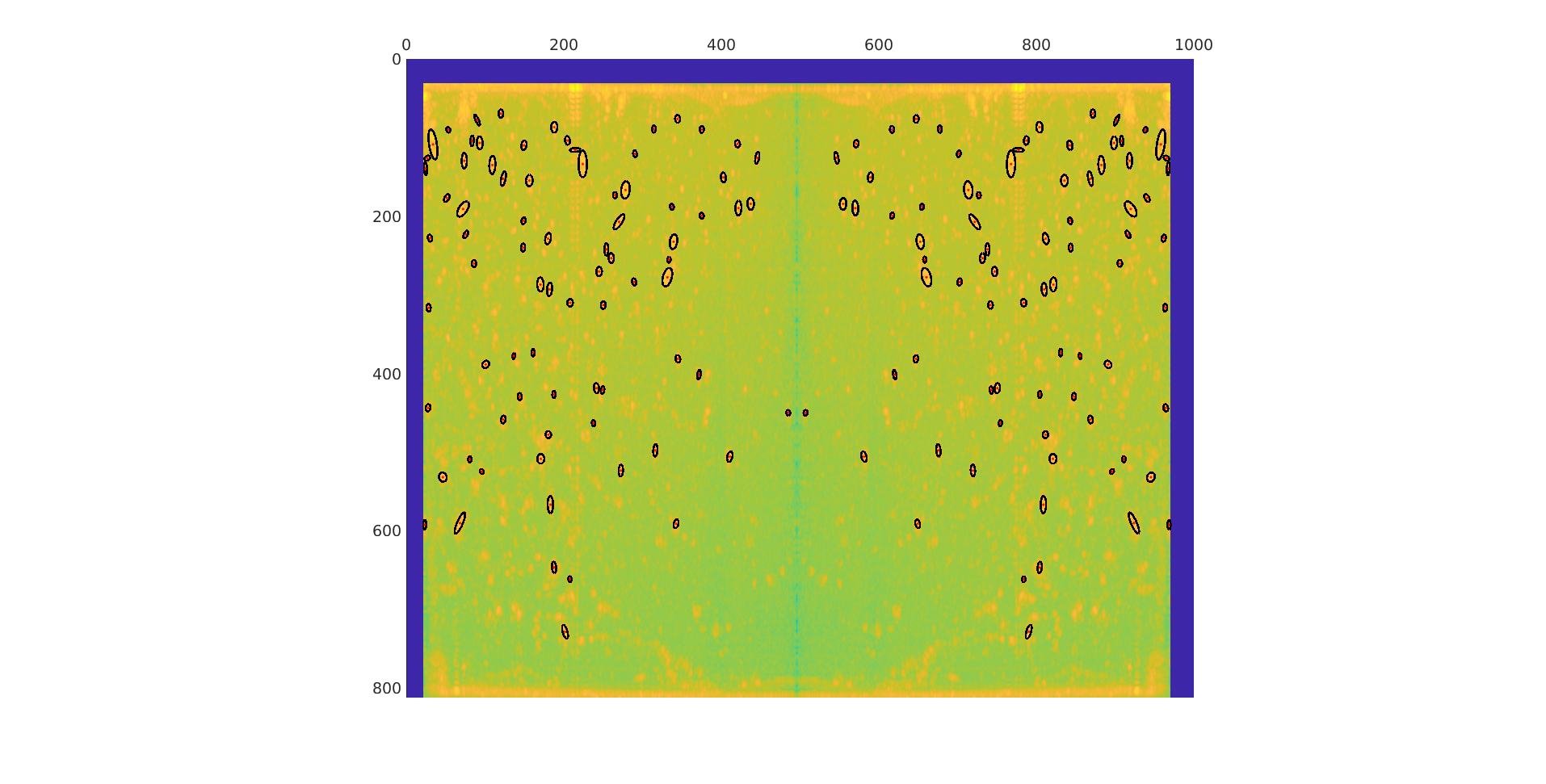} 
	\quad
	\includegraphics[width=0.3\textwidth, clip=true, trim = {12cm 2cm 10cm 0cm}]{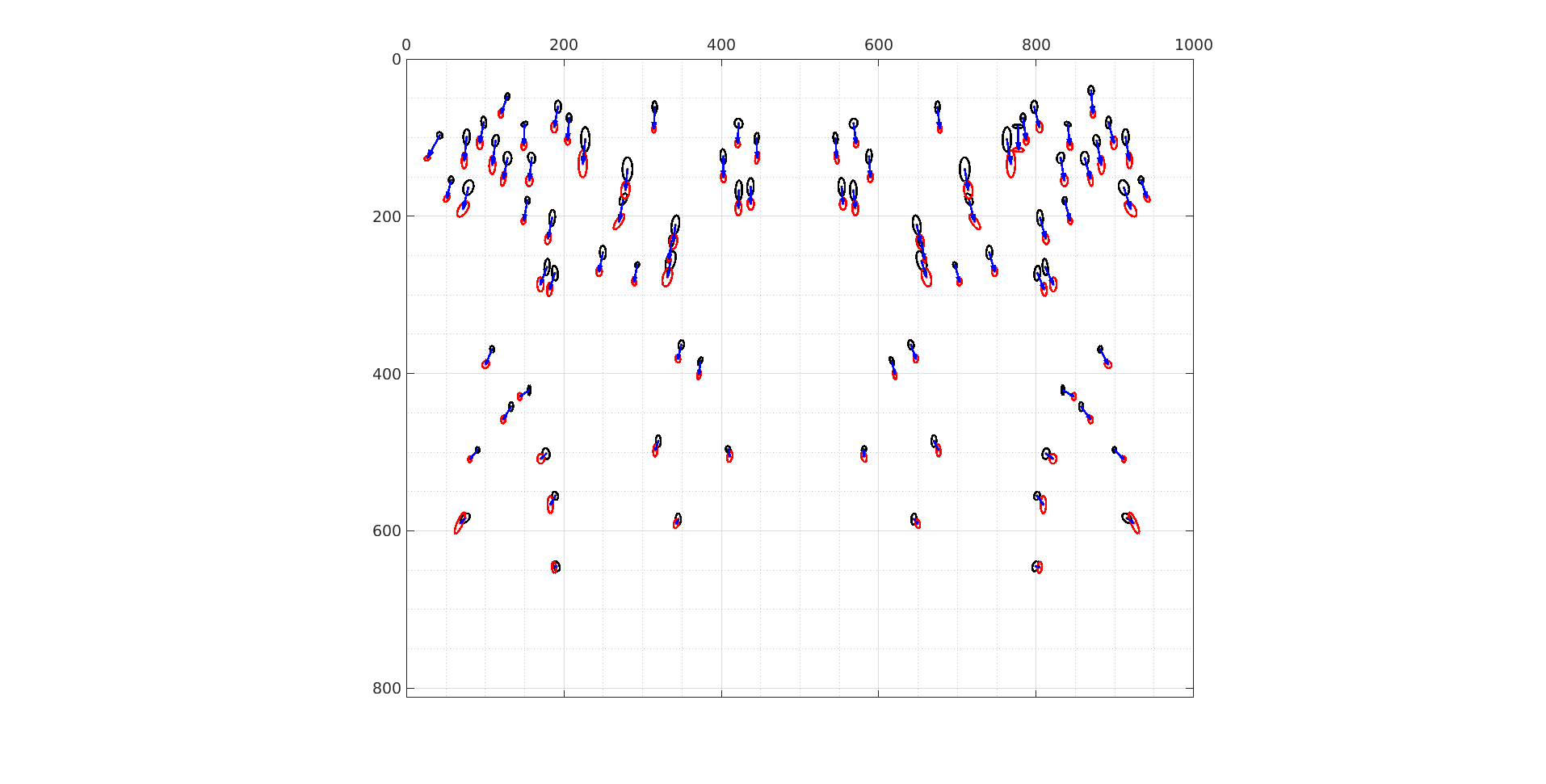}
	\caption{Bubbles (circled in black) detected in the OCT images before compression (left) and after compression (middle) using the approach of Section~\ref{sect_bubble}, together with the resulting bubble displacement (right). Rotational maximum intensity projections. }
	\label{fig_OCT_speckle_tracking}
\end{figure}

Next, we apply our proposed algorithm to the complete 3D volumetric OCT data set, and obtain the results depicted in Figure~\ref{fig_OCT_displacement_rotation}. Many of the brightest bubbles were identified and tracked, resulting in a set of displacement vectors, which by themselves already form a physically plausible displacement field. In fact, using the rotational symmetry of the sample, and after applying a suitable interpolation, a complete displacement field can be obtained, which is already good enough for many further applications.

\begin{figure}[h!] 
	\centering
	\includegraphics[width=0.45\textwidth, clip=true, trim = {10cm 1cm 10cm 2cm}]{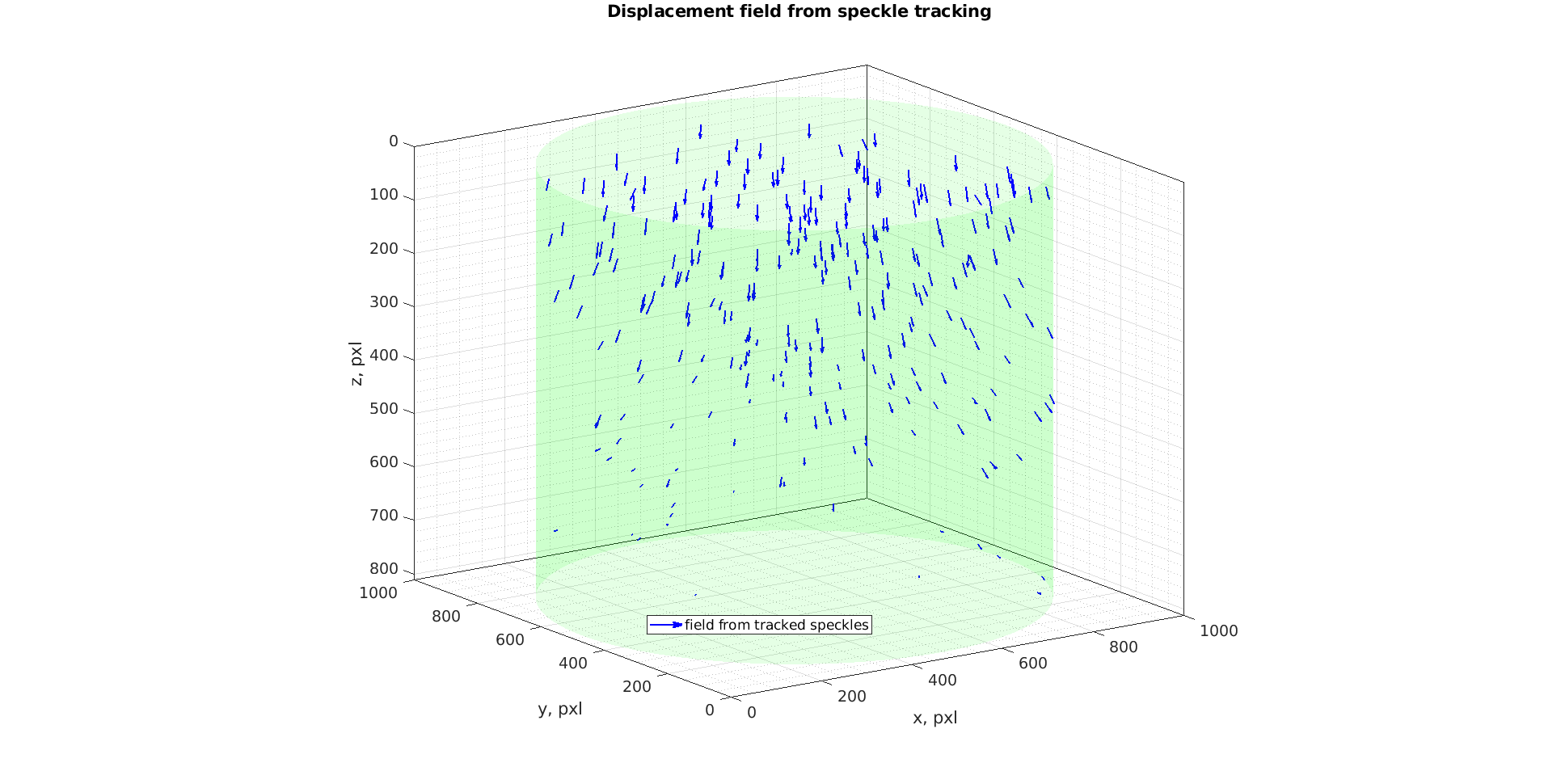}
	\qquad
	\includegraphics[width=0.45\textwidth, clip=true, trim = {10cm 1cm 12cm 1cm}]{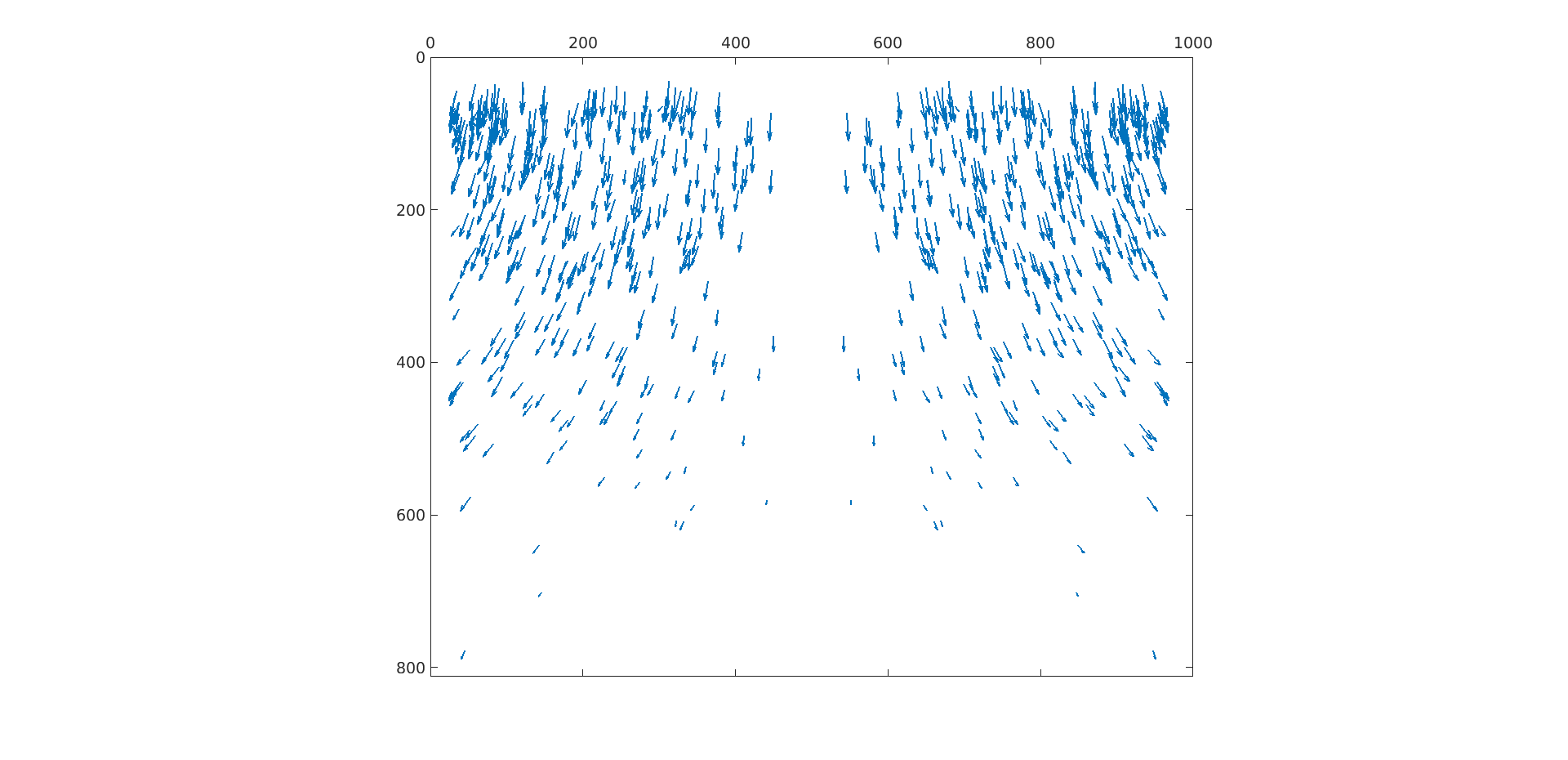}
	\caption{Tracked bubbles (left) and rotational projection (right).} 
	\label{fig_OCT_displacement_rotation}
\end{figure}

We now apply our displacement field estimation algorithm to the images depicted in Figure~\ref{fig_OCT_speckle}, using the same settings and parameters as for the simulated data considered in the previous section, but now with $6$ instead of $5$ layers in the multi-scale approach, since the images are of a higher resolution. When utilizing the bubble information, we use the displacement vectors of all $672$ tracked bubbles depicted in Figure~\ref{fig_OCT_displacement_rotation}.

The resulting displacement fields are depicted in Figure~\ref{fig_OCT_displacement_methods_comp}, again in the four different combinations of utilizing/not utilizing the bubble information, and using/not using the multi-scale approach. Once more it is apparent that utilizing the bubble information greatly improves the quality of the reconstructed displacement fields. While basically no correct displacement fields could be obtained with the standard Horn-Schunck method ($\beta = 0$), when including the bubble information we obtain physically meaningful fields that fit to expectations also in magnitude of the displacement. The positive influence of the multi-scale approach, although not too pronounced in this case, again shows itself in the increased uniformity of the obtained field.

Finally, we investigate the influence of the parameter $\alpha$ on the reconstruction. Figure~\ref{fig_OCT_displacement_alpha_comp} depicts the obtained results for the exact same setting and parameters as before, but now for two different values of $\alpha$. As one may expect, a higher value of $\alpha$ produces a smoother and more uniform displacement field, which is useful in physical situations where one expects the internal displacement field to be smooth anyways.

\begin{figure}[h!] 
	\centering
	\includegraphics[width=0.45\textwidth, clip = true, trim={11cm 3cm 10cm 2cm}]{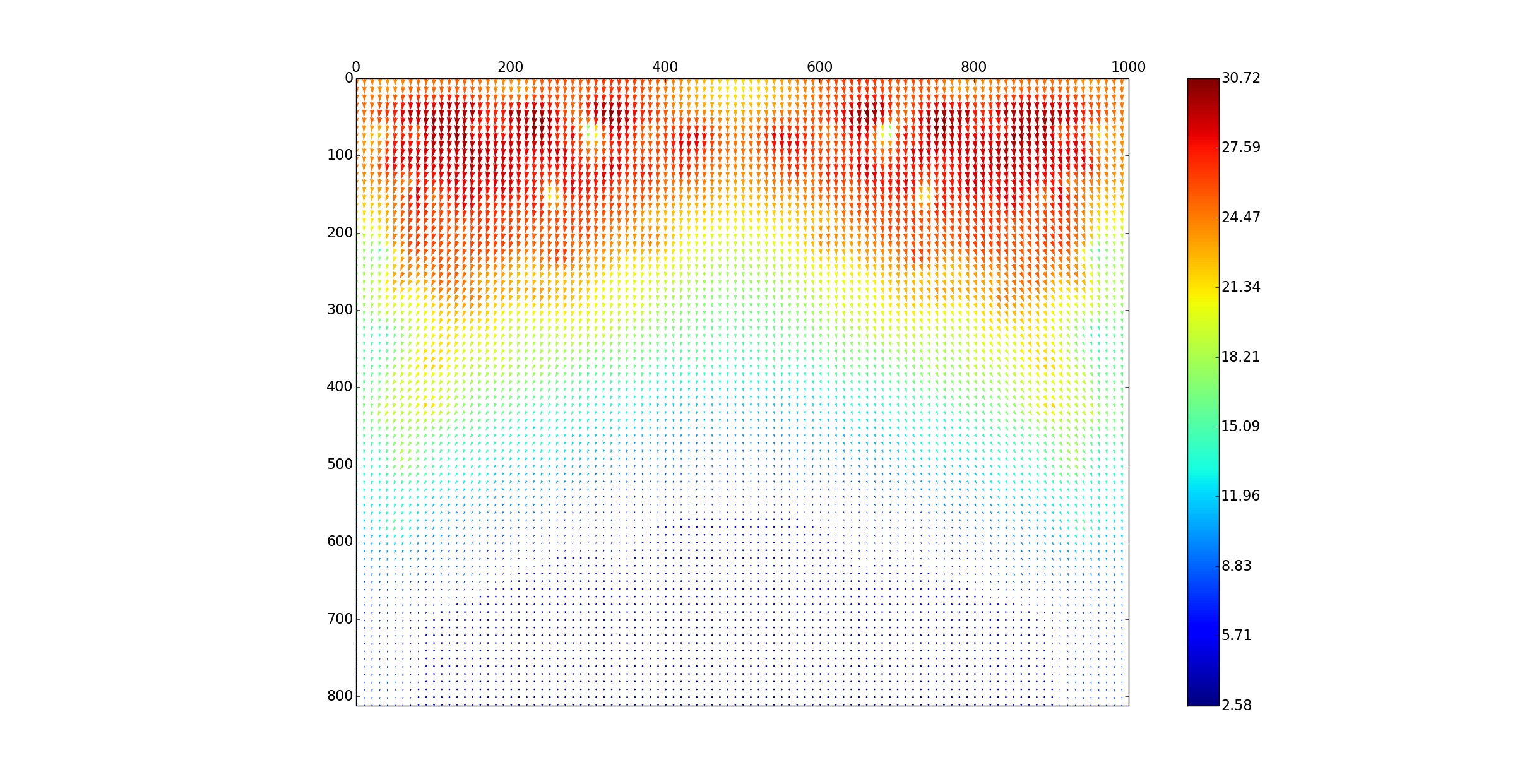}
	\qquad
	\includegraphics[width=0.45\textwidth, clip = true, trim={11cm 3cm 10cm 2cm}]{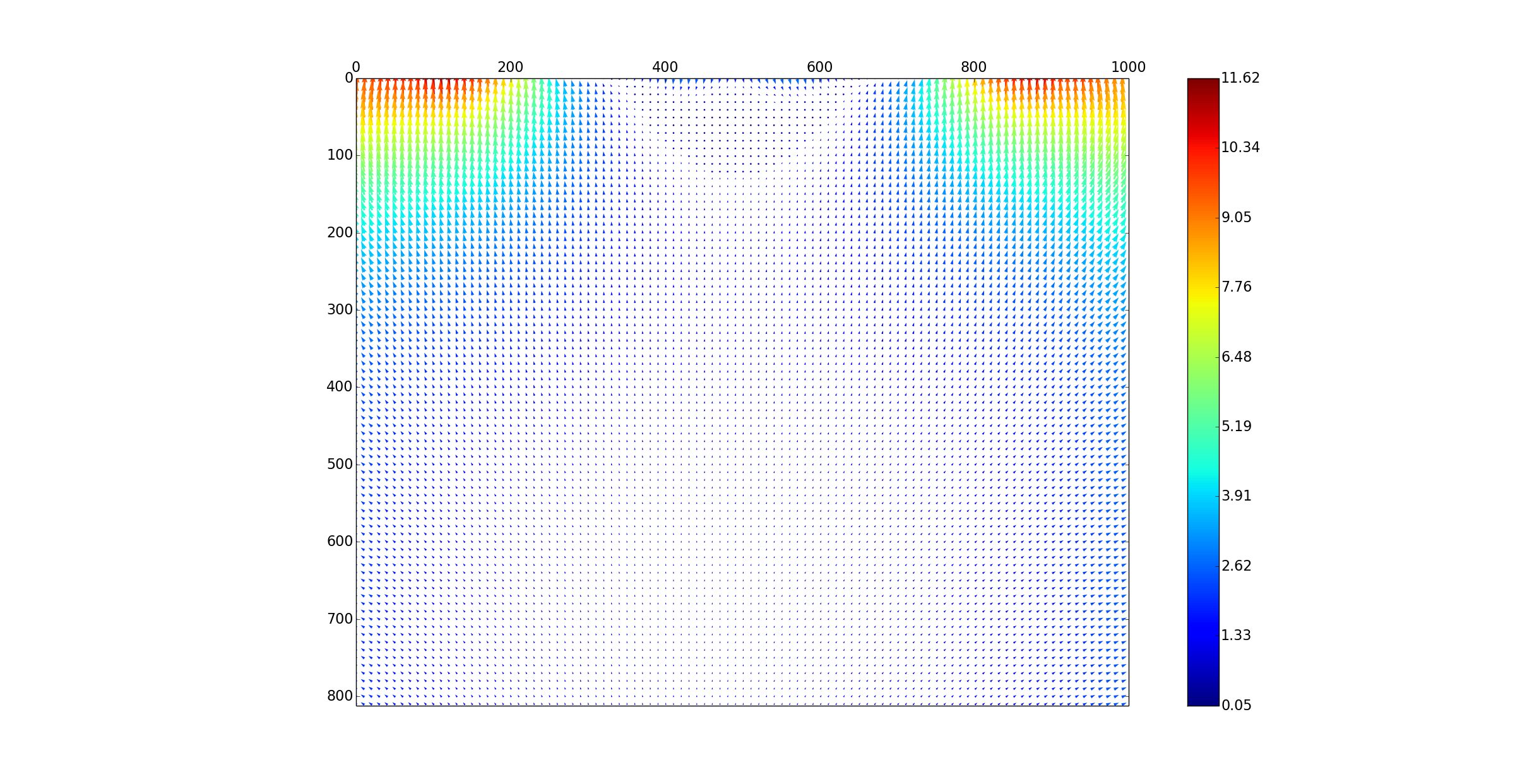}
	\\ \vspace{10pt}
	\includegraphics[width=0.45\textwidth, clip = true, trim={11cm 3cm 10cm 2cm}]{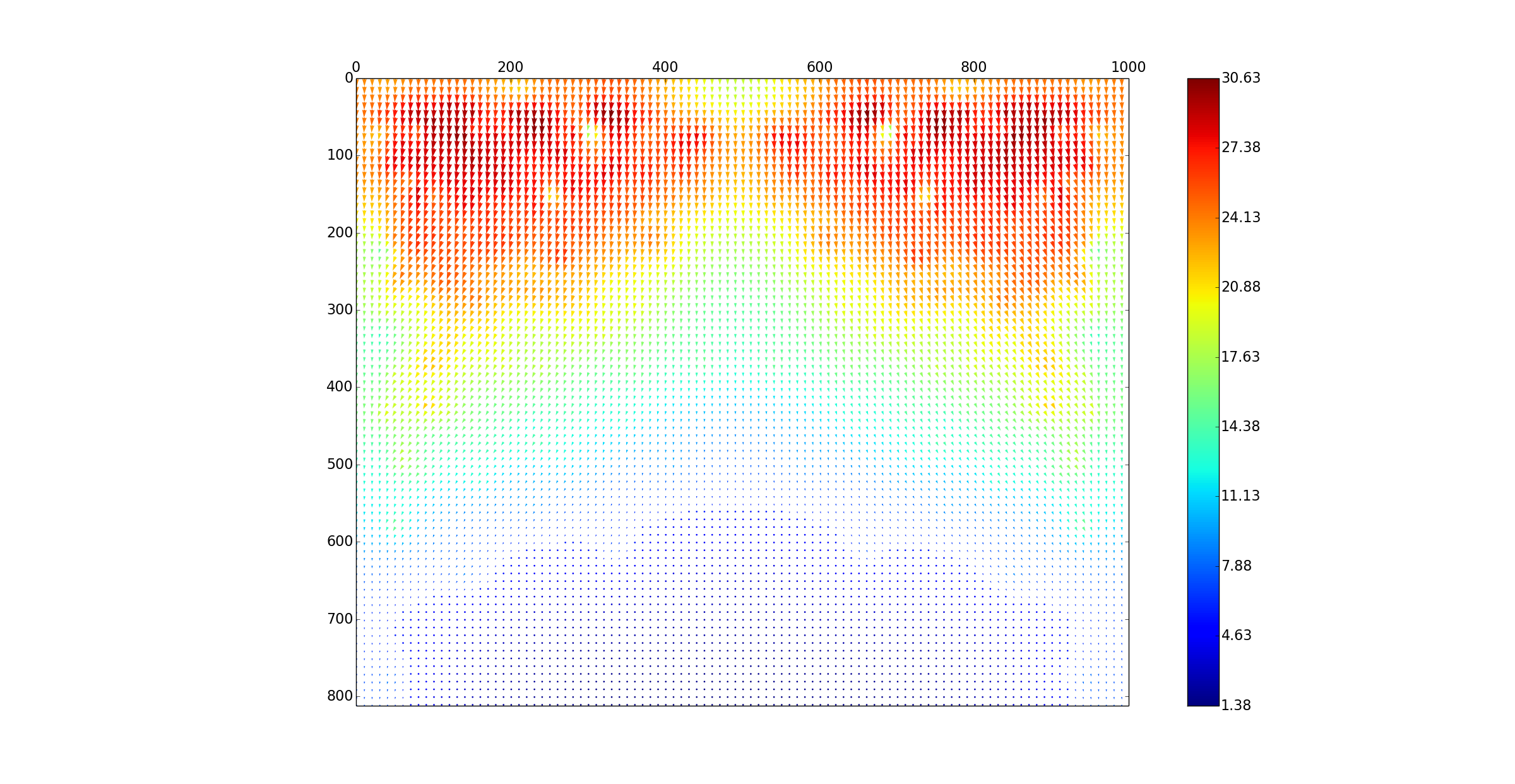}
	\qquad
	\includegraphics[width=0.45\textwidth, clip = true, trim={11cm 3cm 10cm 2cm}]{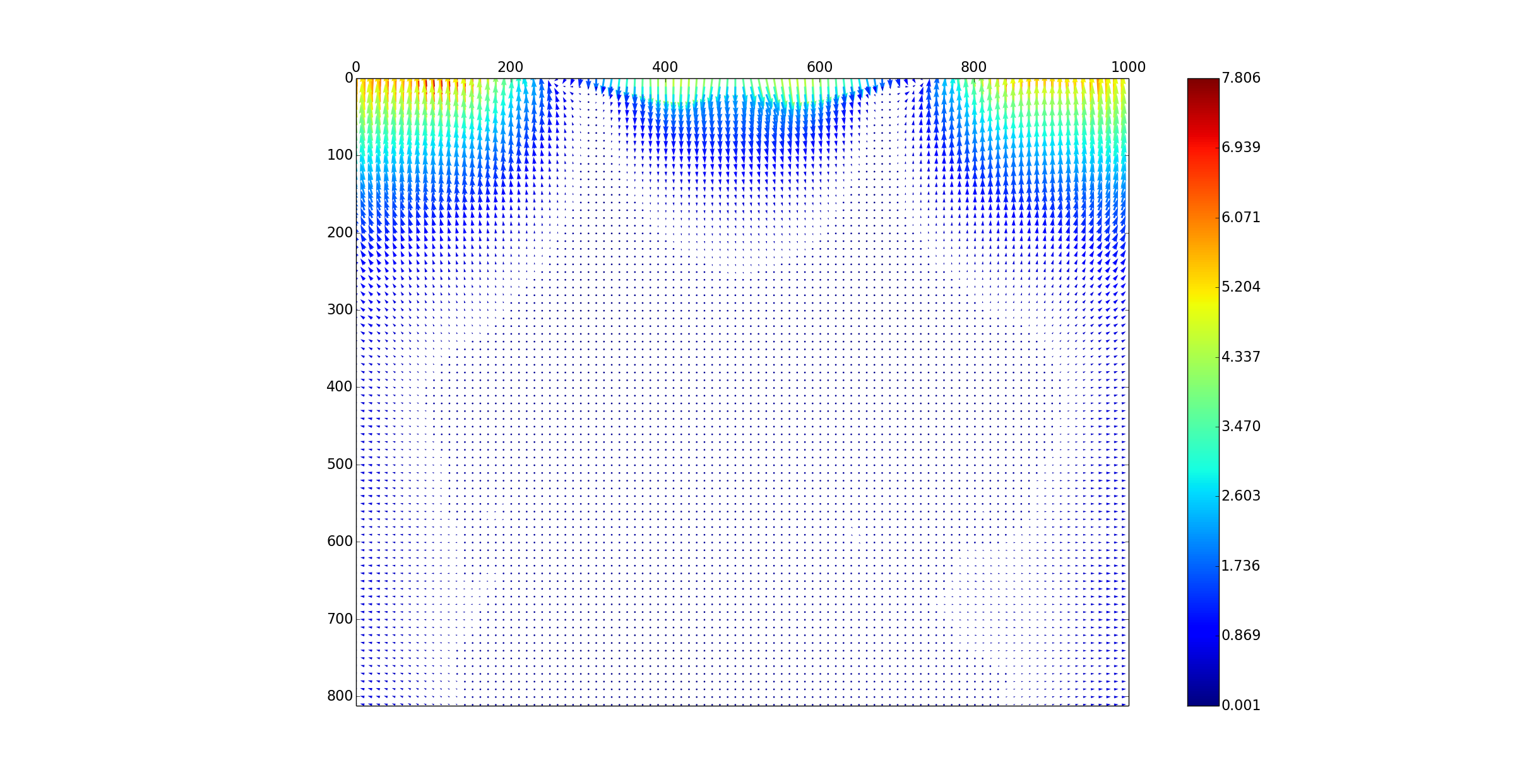}
	\caption{Estimated displacement field for the choices $\beta = 4$ (left) and $\beta = 0$ (right), both with (top) and without (bottom) the use of the multi-scale approach.} 
	\label{fig_OCT_displacement_methods_comp}
\end{figure}

\begin{figure}[h!] 
	\centering
	\includegraphics[width=0.45\textwidth, clip = true, trim={11cm 3cm 10cm 2cm}]{Figures/exp_mult_speckle}
	\qquad
	\includegraphics[width=0.45\textwidth, clip = true, trim={11cm 3cm 10cm 2cm}]{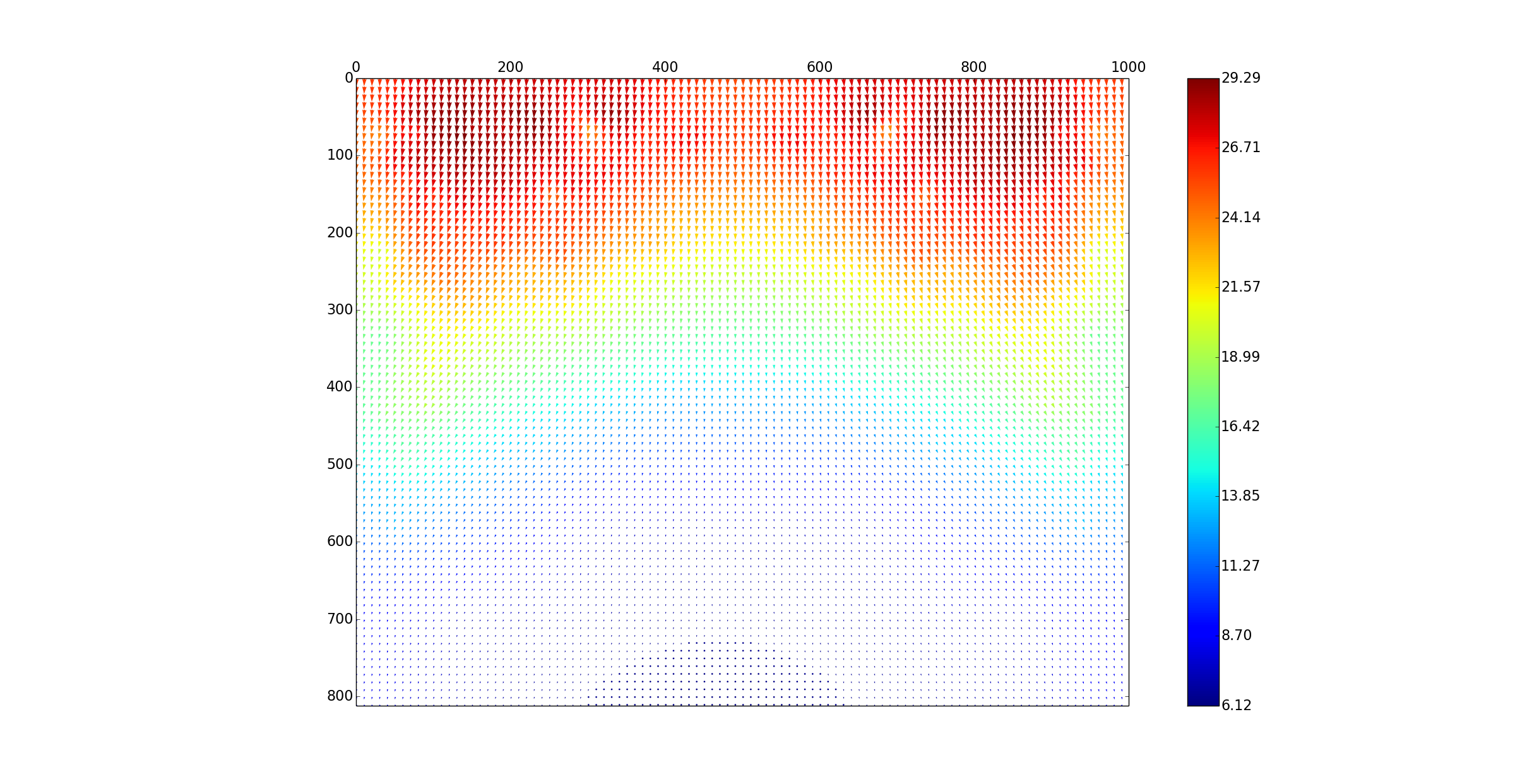}
	\caption{Estimated displacement fields for $\alpha = 0.8$ (left) and $\alpha = 2$ (right).} 
	\label{fig_OCT_displacement_alpha_comp}
\end{figure}

\section{Application to Optical Coherence Elastography}\label{sect_application}

In this section, we consider the application of our proposed displacement field estimation approach to quantitative material parameter estimation via optical coherence elastography. In particular, we are interested in obtaining quantitative estimates of the Young's modulus $E$ of a sample from two OCT measurements before and after compression. The Young's modulus is connected to the Lam\'e parameters $\lambda$ and $\mu$ via
\begin{equation}\label{def_Young}
	E = \frac{\mu(3 \lambda + 2\mu)}{\lambda + \mu} \,,
\end{equation}
whose estimation by iterative regularization methods we recently considered from both an analytical and a numerical viewpoint  in \cite{HubSheNeuSch18}. Note that estimating the Lam\'e parameters $\lambda$ and $\mu$ from a single displacement field in general only allows quantitative estimates of $\mu$, but not of $\lambda$. Reconstructions also of the parameter $\lambda$ were obtained in \cite{HubSheNeuSch18} from simulated data with the same iterative regularization approach as the one presented below, but for Lam\'e parameters with a different contrast and magnitude than the ones considered here. However, the Young's modulus $E$ is not very sensitive to $\lambda$ and thus, as we see below, quantitative estimates of $E$ can still be obtained. If one also wanted to obtain quantitative reconstructions of the parameter $\lambda$, one could for example combine the data resulting from multiple elastography experiments conducted on the same sample, each with a different applied displacement.

In the following, we first restate the mathematical model of linearized elasticity and the corresponding inverse problem of estimating the Lam\'e parameters from a measured displacement field. Afterwards, we recall the iterative regularization approach used in \cite{HubSheNeuSch18} for solving this inverse problem, which we then employ in Section~\ref{sect_numerics_II} to obtain material parameter reconstructions for the simulated and the experimental sample from the displacement field estimations obtained in Section~\ref{sect_numerics} above. Note that of course also other reconstruction methods could be used to extract the Lam\'e parameters from our proposed displacement estimation approach.

\subsection{Mathematical Model of Linearized Elasticity}

Mathematical models linking the material parameters of a sample to the internal and external forces present in an elastography experiment are commonly derived from the equations of motion \cite{Doy12,WidSch15}. Assuming that the sample is linear and isotropic, and that the deformation which the sample undergoes during the experiment is comparatively small, then the model of \emph{linearized elasticity} can be used. It can for example be supplied with mixed Dirchlet and Neumann boundary conditions to model applied displacement and surface traction, respectively.

Mathematically, let $\Omega$ be a non-empty, bounded, open, connected set in $\R^N$, for $N=1,2,3$, with a Lipschitz continuous boundary $\partial \Omega$, which has two subsets $\Gamma_D$ and $\Gamma_T$ such that $\partial \Omega = \overline{\Gamma_D \cup \Gamma_T}$, $\Gamma_D \cap \Gamma_T=\emptyset$ and $\text{meas}(\Gamma_D)>0$. Given body forces $\mathbf{f}$, displacement data $\mathbf{g}_D$, surface traction $\mathbf{g}_T$ and Lam\'e parameters $\lambda$ and $\mu$, the (forward) problem of linearized elasticity with displacement-traction boundary conditions consists in finding the displacement field $\uf$ satisfying
\begin{equation}\label{eq_elast_lin}
	\begin{split}
		-\div{\sigma(\uf)} & = \mathbf{f} \,, \qquad \text{in} \; \Omega \,, \\
		\uf |_{\Gamma_D}   & = \mathbf{g}_D \,, \\
		\sigma(\uf) \vec{n} |_{\Gamma_T} & = \mathbf{g}_T \,,
	\end{split}
\end{equation}
where $\vec{n}$ is the outward unit normal vector of $\partial \Omega$, and the stress tensor $\sigma$ defining the stress-strain relation in $\Omega$ is defined by
\begin{equation*}
	\sigma(\uf):=\lambda \div{\uf} I + 2\mu \mathcal{E}(\uf) \,,
\end{equation*}
where $I$ is the identity matrix and $\mathcal{E}$ is the strain tensor defined by
\begin{equation*}
	\mathcal{E}(\uf) := \frac{1}{2}\left(\nabla \uf + \nabla \uf^T \right) \,.
\end{equation*}
One typically considers the following weak formulation of this problem \cite{Bra07,Cia94,HubSheNeuSch18}: Find a function  $\uf \in V :=H^1_{0,\Gamma}(\Omega)^N$, where $H^1_{0,\Gamma}(\Omega) := \{\uf \in \HoO \, | \, \uf|_{\Gamma_D} =0 \}$, which satisfies
\begin{equation}\label{eq_variat_form}
	a_{\lambda,\mu}(\uf,\vf) = l(\vf)-a_{\lambda,\mu}(\mathbf{\Phi},\vf)\,, \qquad \forall\vf\in V\,,
\end{equation}
where $\mathbf{\Phi}$ is a homogenization function satisfying $\mathbf{\Phi}|_{\Gamma_D}=\mathbf{g}_D$, and the bilinear form $a_{\lambda,\mu}(\cdot,\cdot)$ and the linear form $l(\cdot)$ are given by 
\begin{equation*}
	\begin{split}
		a_{\lambda,\mu}(\uf,\vf) &:=\intVx{\left(\lambda \div{\uf}\div{\vf} + 2\mu \, \mathcal{E}(\uf):\mathcal{E}(\vf)\right)}\,,
		\\
		l(\vf)&:=\spr{\mathbf{f},\vf}_{H^{-1}(\Omega),\HoO} + \spr{\mathbf{g}_T,\vf}_{H^{-\frac{1}{2}}(\Gamma_T),H^{-\frac{1}{2}}(\Gamma_T)} \,.
	\end{split}
\end{equation*}
If for example $\mathbf{f}\in H^{-1}(\Omega)^{N}$, $\mathbf{g}_D \in H^{\frac{1}{2}}(\Gamma_D)^N$, $\mathbf{g}_T \in H^{-\frac{1}{2}}(\Gamma_T)^{N}$, and $\mathbf{\Phi} \in \HoO^N$, then the variational problem \eqref{eq_variat_form} admits a unique solution \cite{Bra07,Cia94,HubSheNeuSch18}.

\subsection{The Inverse Problem}

After considering the forward problem of linearized elasticity in the previous section, we now turn our attention to the inverse problem, which consists in determining the Lam\'e parameters $\lambda$ and $\mu$ from measurements of the displacement field $\uf$. Introducing for $s>N/2+1$ the \emph{parameter-to-solution map} $F$, i.e., the nonlinear operator
\begin{equation*}
	\begin{split}
		F:\mathcal{D}(F):=\{(\lambda,\mu)\in H^s(\Omega)^2 \, \vert \, \lambda\ge 0, \mu \ge \underline{\mu} >0 \} & \rightarrow \LtO^N\,, \\
		(\lambda,\mu) & \mapsto \uf(\lambda,\mu) \,,
	\end{split}
\end{equation*}
where $\uf$ is the solution of the variational problem \eqref{eq_variat_form}, the inverse problem can be formulated as the problem of solving the nonlinear operator equation
\begin{equation}\label{eq_operator_eq}
	F(\lambda,\mu) = \uf \,.
\end{equation}
In practical applications, instead of the true displacement field $\uf$ one only has access to noisy data/measurements $\uf^\delta$, which for example satisfy the estimate 
\begin{equation}\label{cond_noise}
	\norm{\uf - \uf^\delta}_{\LtO} \le \delta \,,
\end{equation}
where $\delta \ge 0$ denotes the noise level. Since the inverse problem is ill-posed and thus very sensitive even to a small amount of noise in the data, methods for obtaining (approximate) solutions of \eqref{eq_operator_eq} need to be carefully chosen. One such choice, an iterative regularization approach based on Landweber iteration, which was used for obtaining the numerical results presented in Section~\ref{sect_numerics_II}, is discussed below.

\subsection{Iterative Regularization Approach}\label{sect_it_reg_approach}

One of the most well known iterative regularization methods for solving nonlinear inverse problems is \emph{Landweber iteration} \cite{EngHanNeu96,KalNeuSch08}, given by
\begin{equation}\label{eq_Landweber}
	(\lambda_{k+1},\mu_{k+1}) = (\lambda_{k},\mu_{k}) - \omega_{k}^\delta F'(\lambda_{k},\mu_{k})^*(F(\lambda_{k},\mu_{k})-\uf^\delta)\,,
\end{equation}
where $k$ is the iteration index and $\omega_{k}^\delta$ is a sequence of stepsizes. As a stopping rule one typically uses the \emph{discrepancy principle}, which determines the stopping index $k_*$ as the smallest index such that 
\begin{equation}\label{eq_discrepancy}
	\norm{F(\lambda_{k},\mu_{k})-\uf^\delta} \leq \tau \delta\,,
\end{equation}
where $\delta$ is as in \eqref{cond_noise} and $\tau > 1$ is a constant. For the stepsizes $\omega_{k}^\delta$ one can for example use a suitably chosen constant, the minimal error, or the steepest descent stepsize \cite{Sch96}, the latter one being given by
\begin{equation}\label{def_okd}
	\omega_{k}^\delta := \frac{\norm{s_k^\delta}^2}{\norm{F'(\lambda_k,\mu_k)^*s_k^\delta }^2}  \,,
	\qquad
	s_k^\delta := F'(\lambda_{k},\mu_{k})^*(F(\lambda_{k},\mu_{k})-\uf^\delta)\,.
\end{equation}
In order to guarantee convergence of this, and in fact most other iterative regularization methods, one typically has to use the so-called \emph{tangential cone condition} or one of its weaker variants, which are often very difficult to show for non-academic examples. Fortunately, for our present problem, if one knows the Lam\'e parameters in an (arbitrarily) small neighbourhood of the boundary and instead of \eqref{eq_operator_eq} considers 
\begin{equation*}
	F_c(\lambda,\mu) = \uf \,,
\end{equation*}
where the operator $F_c$ incorporates the knowledge of the Lam\'e parameters near the boundary, then even the strong tangential cone condition holds. This guarantees the convergence of \eqref{eq_Landweber} with $F$ replaced by $F_c$ when combined with the discrepancy principle \eqref{eq_discrepancy}. For more details on this matter we refer the reader to \cite{HubSheNeuSch18}.

In order to implement the iteration \eqref{eq_Landweber} and the stepsize \eqref{def_okd}, one needs explicit expressions for the Fr\'echet derivative $F'(\lambda,\mu)(h_\lambda,h_\mu)$ and its adjoint $F'(\lambda,\mu)^*\mathbf{w}$. These have been given in \cite{HubSheNeuSch18} and involve the solution of a number of large-scale variational problems, which can be very costly and time consuming. Even though parallelisation can be used to reduce the computational time required for solving these variational problems, the convergence of \eqref{eq_Landweber} may still be too slow. Hence, in order to speed up the iteration, we employ Nesterov's acceleration principle, which leads to the iteration
\begin{equation}\label{eq_Nesterov}
	\begin{split}
		(\bar{\lambda}_{k},\bar{\mu}_{k}) &= (\lambda_{k},\mu_{k}) + \alpha_{k}^\delta \kl{(\lambda_{k},\mu_{k}) - (\lambda_{k-1},\mu_{k-1})} \,,
		\\
		(\lambda_{k+1},\mu_{k+1}) &= (\bar{\lambda}_{k},\bar{\mu}_{k}) - \omega_{k}^\delta F'(\bar{\lambda}_{k},\bar{\mu}_{k})^*(F(\bar{\lambda}_{k},\bar{\mu}_{k})-\uf^\delta)\,,
	\end{split}
\end{equation}
where also $\omega_{k}^\delta$ is computed using the intermediate iterates $(\bar{\lambda}_{k},\bar{\mu}_{k})$. In the numerical results presented below, we used the popular choice of $\alpha_{k}^\delta := (k-1)/(k+2)$; see for example \cite{Jin16}. For the solution of nonlinear inverse problems, iteration methods of the form \eqref{eq_Nesterov} have been studied for various choices of the parameters $\alpha_{k}^\delta$ in \cite{HubRam17, HubRam18}, and convergence was established essentially under the same tangential cone condition as needed for the convergence of Landweber iteration.

In case the noise level $\delta$ is unknown or estimates of it are unreliable, then instead of the discrepancy principle one can use heuristic stopping rules such as the heuristic discrepancy principle \cite{Kin11, KinRai19}, which determines the stopping index $k_*$ by minimizing
\begin{equation*}
	k \mapsto \sqrt{k} \norm{F\kl{\lambda_{k},\mu_{k}}} \,,
\end{equation*} 
and has been successfully employed to a wide range of both linear and nonlinear inverse problems. Even if the noise level $\delta$ is known, using the discrepancy principle can remain a difficult task due to the freedom which one has in choosing the parameter $\tau$. Popular choices such as $\tau=1.1$ are often not feasible, for example when the residual cannot be brought to the level of $\tau \delta$ (in an acceptable amount of time) in practice. Furthermore, theoretically backed choices depend on parameters that are typically either unknown or overestimated  \cite{KalNeuSch08}. Hence, for the numerical results presented below, we used the (heuristic) discrepancy principle together with a monitoring of the residual and the reconstruction quality as an informed guideline to manually determine a suitable stopping index for each test.

It was shown in \cite[Theorem~3.4]{HubSheNeuSch18} that the operator $F'(\lambda,\mu)^*$ can be written as the composition of two operators $E_s$ and $A$, where $E_s$ basically corresponds to the adjoint of the embedding operator from $H^s(\Omega)^2$ to $\LtO^2$. The numerical examples presented in \cite{HubSheNeuSch18} showed that dropping the operator $E_s$ in iteration \eqref{eq_Landweber} or \eqref{eq_Nesterov} leads to improved Lam\'e parameter reconstructions, in which larger jumps are better resolved. Mathematically, this can be understood as a (Hilbert scale) preconditioning of the problem; see for example \cite{EngHanNeu96}. Hence, for the numerical examples presented below, we always dropped the operator $E_s$ in our iterative reconstruction approach \eqref{eq_Nesterov}.

\section{Numerical Results II: Parameter Estimation}\label{sect_numerics_II}

In this section, we employ the iterative regularization approach discussed in Section~\ref{sect_it_reg_approach} to obtain material parameter reconstructions for the simulated and the experimental sample from the displacement field estimations obtained in Section~\ref{sect_numerics} above.

\subsection{Simulated Data}

In this section, we consider the simulated elastography setup from Section~\ref{sect_sim_exp}. The results of our iterative reconstruction approach using the operator $F$, applied to both the exact and the estimated displacement field (see Figure~\ref{fig_sim_exp_fields}), are depicted in Figure~\ref{fig_sim_exp_rec_1}. Starting from the uniform initial guess $(\lambda_0,\mu_0) = (490,10)$, the iteration was stopped after $85$ and $27$ iterations taking about $30$ and $10$ minutes, respectively, and the Young's modulus $E$ was computed from the reconstructed Lam\'e parameters via \eqref{def_Young}. Comparing the results with Figure~\ref{fig_sim_exp_exact} one can clearly see that the reconstructed parameters $\mu$ and $E$ agree well with the exact parameters both in shape and in magnitude. Even though, as noted above, the relative error in the estimated displacement field is $11.53\%$, we still manage to obtain quantitative parameter reconstructions. In particular, the difficulty in recovering the Lam\'e parameter $\lambda$ from just a single displacement field does not prohibit a quantitative reconstruction of the Young's modulus~$E$.

\begin{figure}[ht!]
	\centering
	\includegraphics[width=1\textwidth, clip = true, trim={0cm 0cm 0cm 0cm}]{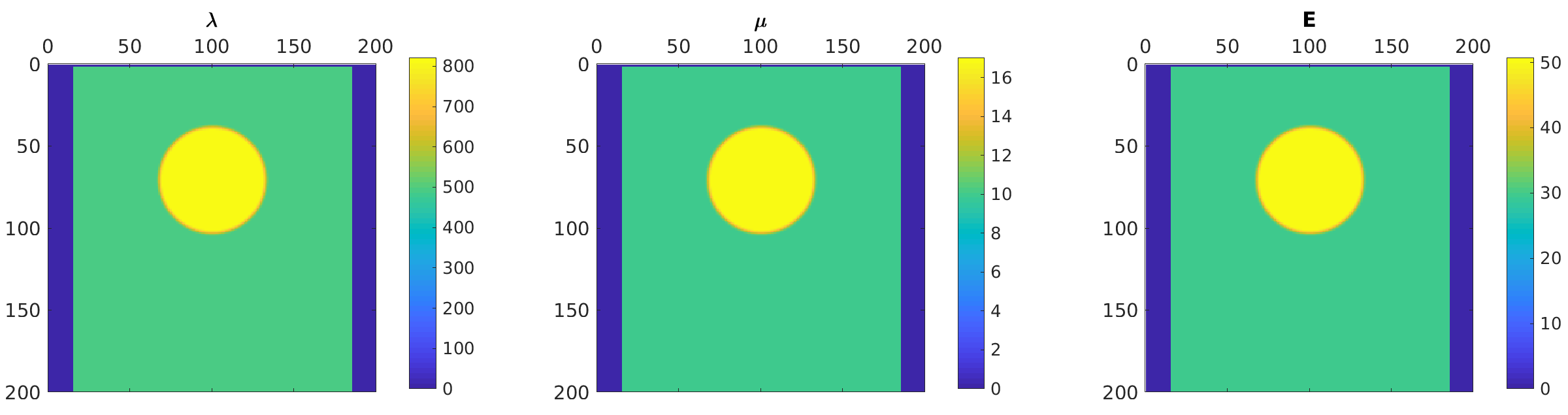}
	\caption{Exact Lam\'e parameters $\lambda$ (left) and $\mu$ (middle), together with the exact Young's modulus $E$ (right) for the simulated sample depicted in Figure~\ref{fig_sim_exp_sample_field}.}
	\label{fig_sim_exp_exact}
\end{figure}

\begin{figure}[ht!]
	\centering
	\includegraphics[width=1\textwidth, clip = true, trim={0cm 0cm 0cm 0cm}]{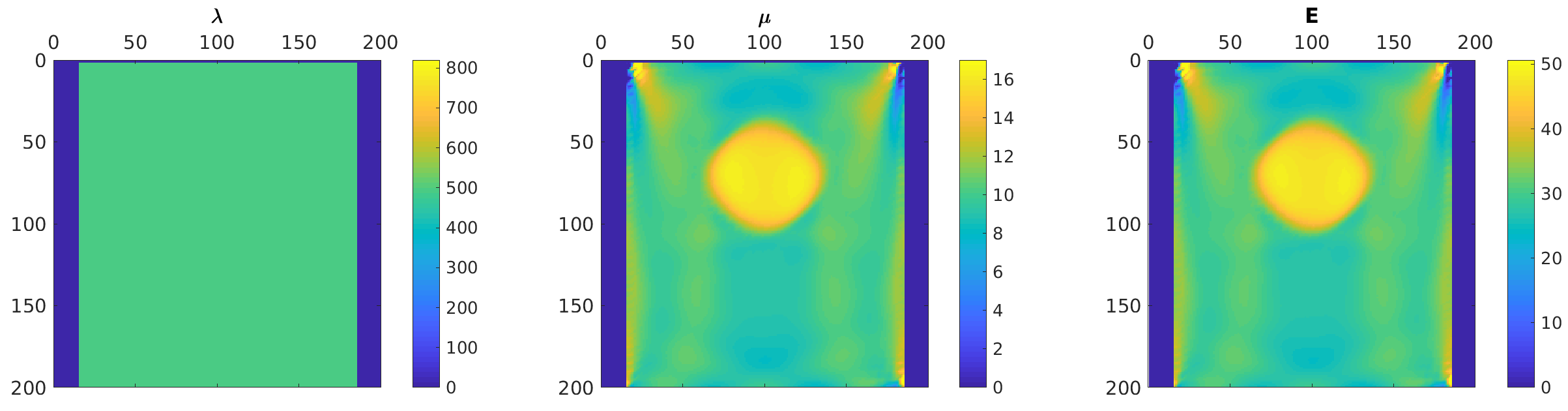}
	\\
	\includegraphics[width=1\textwidth, clip = true, trim={0cm 0cm 0cm 0cm}]{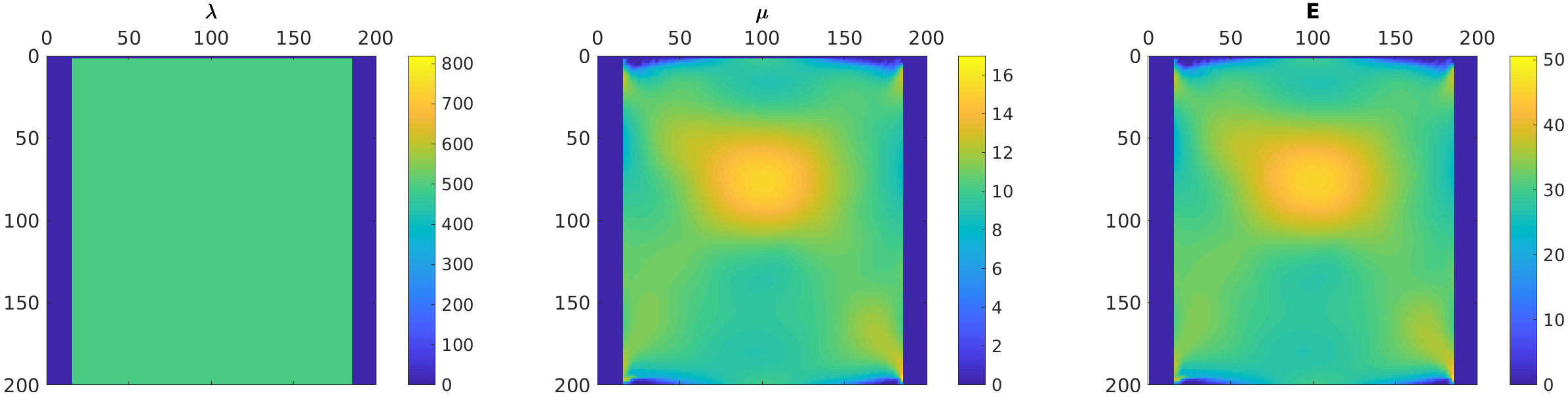}
	\caption{Reconstructed Lam\'e parameters $\lambda$ (left) and $\mu$ (middle), and the Young's modulus $E$ (right) for the simulated sample depicted in Figure~\ref{fig_sim_exp_sample_field}. The reconstructions were obtained by applying the iterative reconstruction approach presented in Section~\ref{sect_it_reg_approach} using the operator $F$ to the exact (top) and estimated (bottom) displacement field. Concerning details on the inability to reconstruct the parameter $\lambda$ see the main text.}
	\label{fig_sim_exp_rec_1}
\end{figure}

Next, we assume that the Lam\'e parameters $\lambda$ and $\mu$ are known in a small neighbourhood of the boundary of the sample. Hence, we can use the operator $F_c$ instead of the operator $F$ in our iterative reconstruction approach, which we again apply to both the exact and the estimated displacement field. The results after $88$ and $3$ iterations taking about $30$ and $1$ minute, respectively, are depicted in Figure~\ref{fig_sim_exp_rec_2}. As before, we obtain quantitative reconstructions of the Lam\'e parameter $\mu$ and the Young's modulus $E$, which are not deteriorated by the inability to reconstruct $\lambda$ from only a single displacement field. Note that the obtained parameter reconstructions using the operator $F_c$, in particular the ones derived from the exact displacement field, better resolve the (circular) shape of the inclusion.

\begin{figure}[ht!]
	\centering
	\includegraphics[width=1\textwidth, clip = true, trim={0cm 0cm 0cm 0cm}]{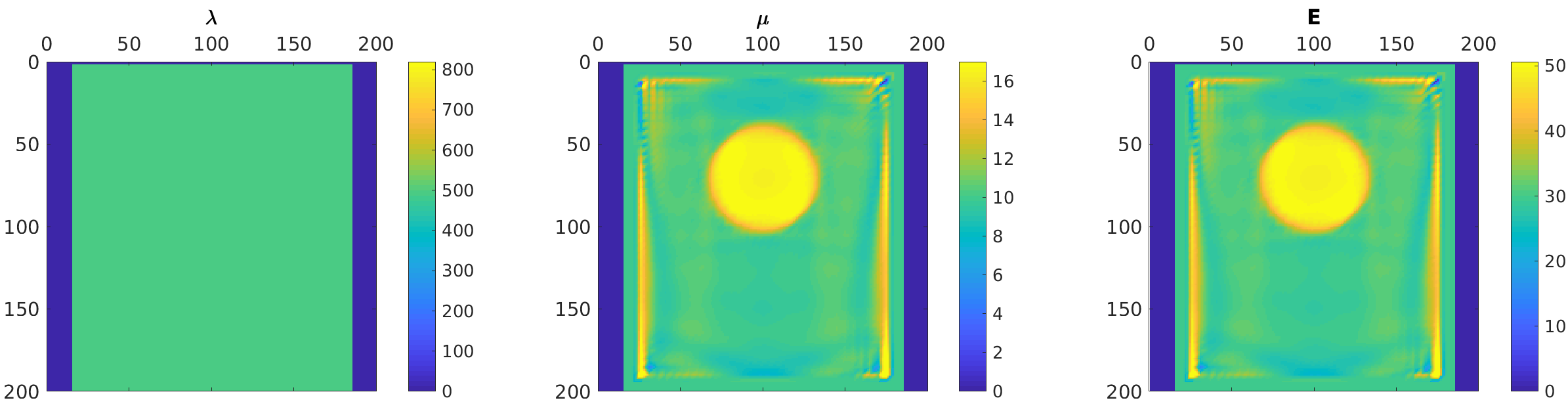}
	\\
	\includegraphics[width=1\textwidth, clip = true, trim={0cm 0cm 0cm 0cm}]{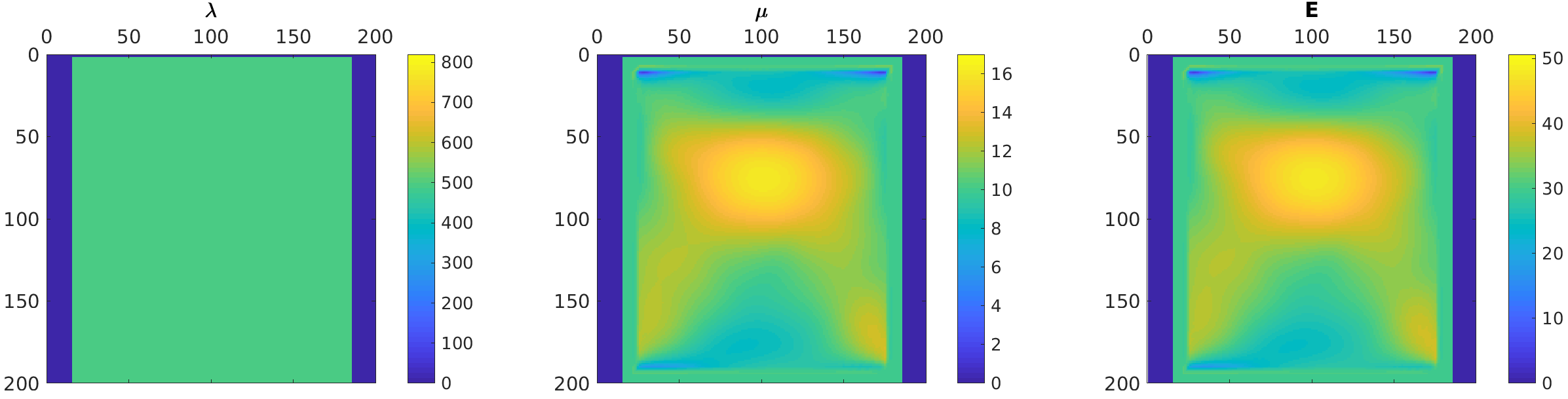}
	\caption{Reconstructed Lam\'e parameters $\lambda$ (left) and $\mu$ (middle), and the Young's modulus $E$ (right) for the simulated sample depicted in Figure~\ref{fig_sim_exp_sample_field}. The reconstructions were obtained by applying the iterative reconstruction approach presented in Section~\ref{sect_it_reg_approach} using the operator $F_c$ to the exact (top) and estimated (bottom) displacement field.}
	\label{fig_sim_exp_rec_2}
\end{figure}

\subsection{Experimental Data}

After considering simulated data in the previous section, we now turn our attention to real data stemming from an OCT elastography experiment. In particular, we apply our iterative reconstruction approach to the homogeneous sample for which we have already estimated the displacement field in Section~\ref{sect_numerics} above. 

The resulting reconstructions of the Lam\'e parameters $\lambda$ and $\mu$ as well as the Young's modulus $E$, obtained after $10$ iterations taking about $10$ minutes and given in kPA, are depicted in Figure~\ref{fig_exp_rec}. One can see that in the inner part of the sample a relatively uniform parameter reconstruction, in particular for $\mu$ and $E$, was obtained. The steep rise towards the edges of the samples is either an artefact due to the dropping of the (smoothing) operator $E_s$ in the iterative reconstruction method, or caused by the fact that the estimated displacement field becomes somewhat inaccurate towards the sides of the sample due to the lack of detected bubbles. Computing the mean values over the inner part of the sample (pixels $300$ to $700$ in x-direction and $200$ to $600$ in y-direction), we get the parameter estimates $\mu = 342$ kPA and $E = 1015$ kPA, with a standard deviation of $49$ kPA and $143$ kPA, respectively.


We currently investigate the efficiency and accuracy of our displacement field estimation and material parameter reconstruction methods also on inhomogeneous samples.

\begin{figure}[ht!]
	\centering
	\includegraphics[width=0.32\textwidth, clip = true, trim={0cm 0cm 0cm 0cm}]{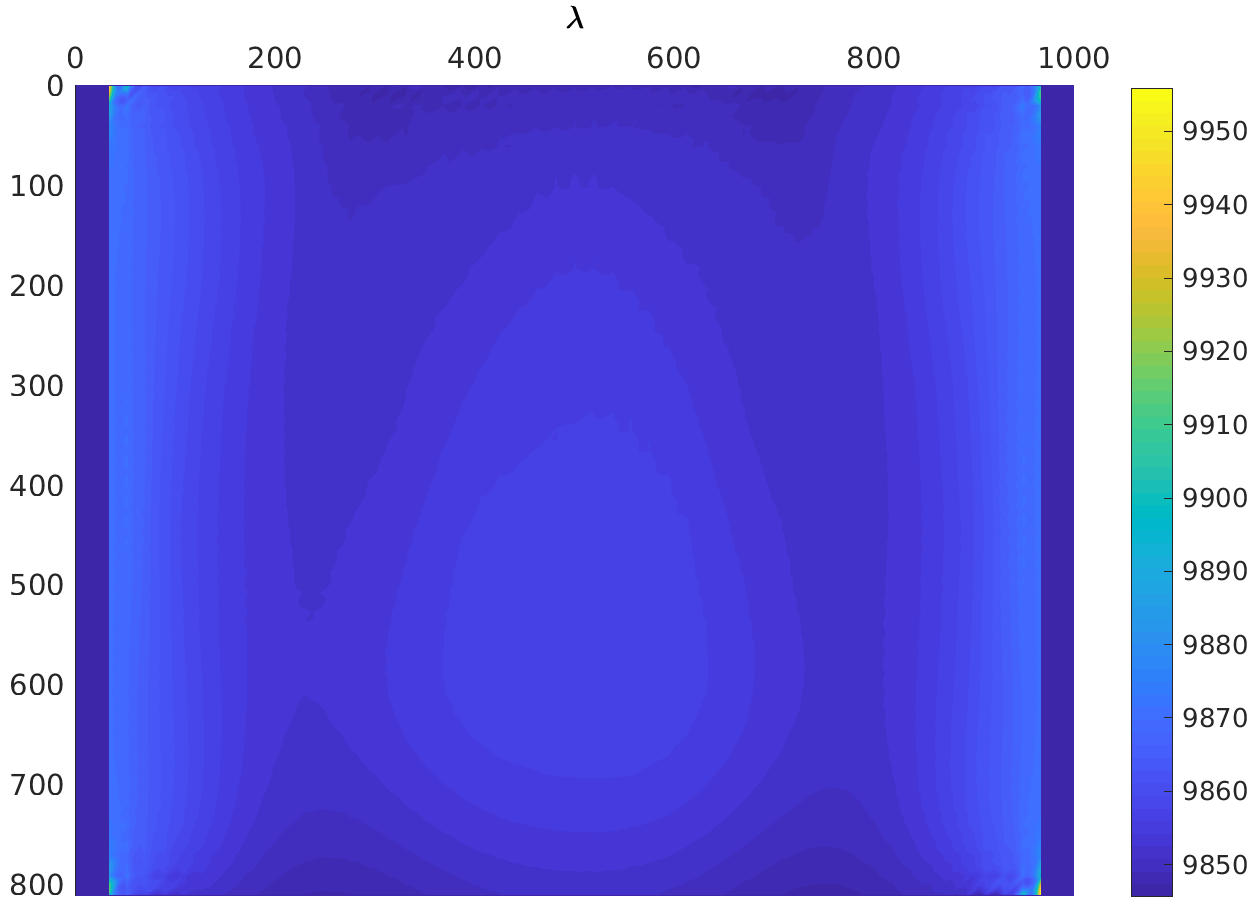}
	\,
	\includegraphics[width=0.32\textwidth, clip = true, trim={0cm 0cm 0cm 0cm}]{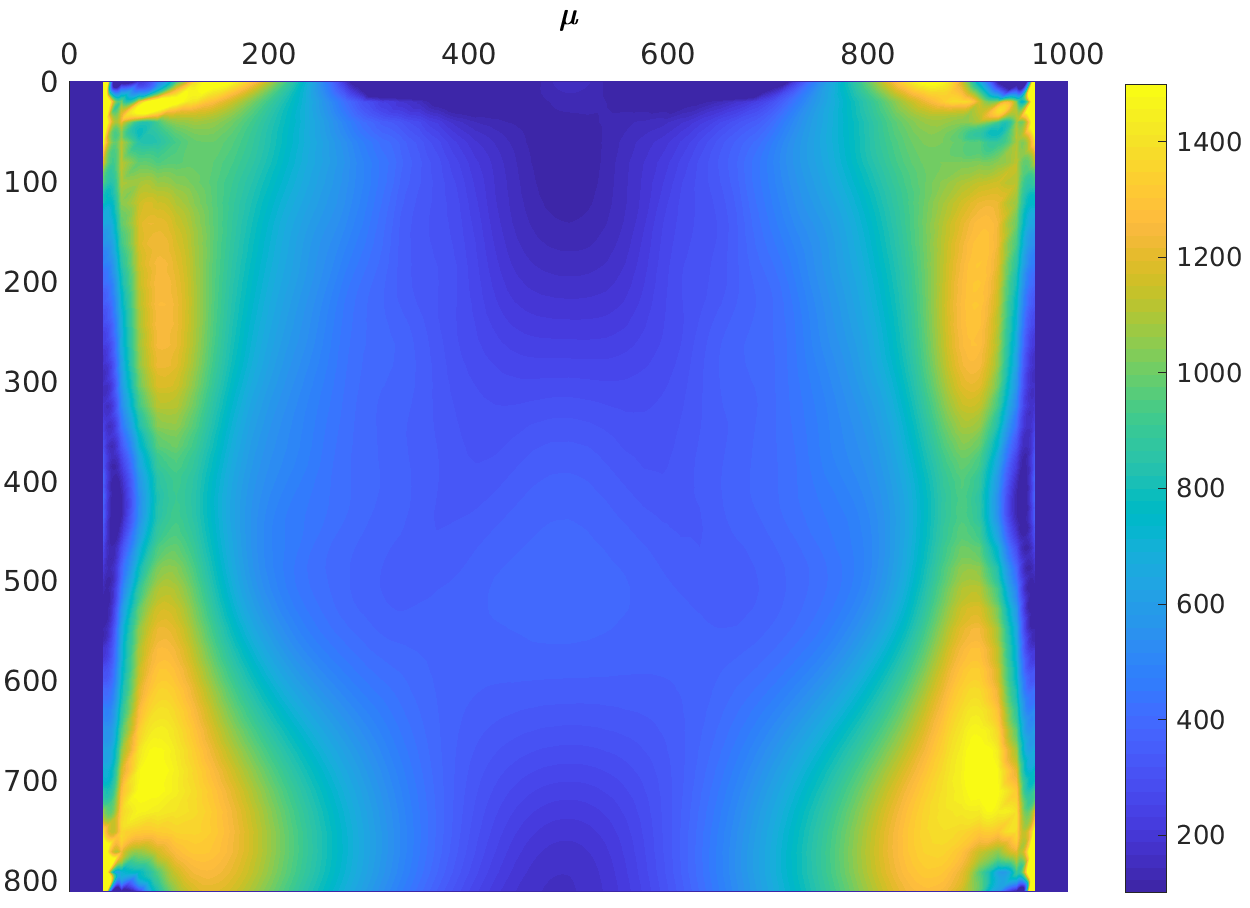}
	\,
	\includegraphics[width=0.32\textwidth, clip = true, trim={0cm 0cm 0cm 0cm}]{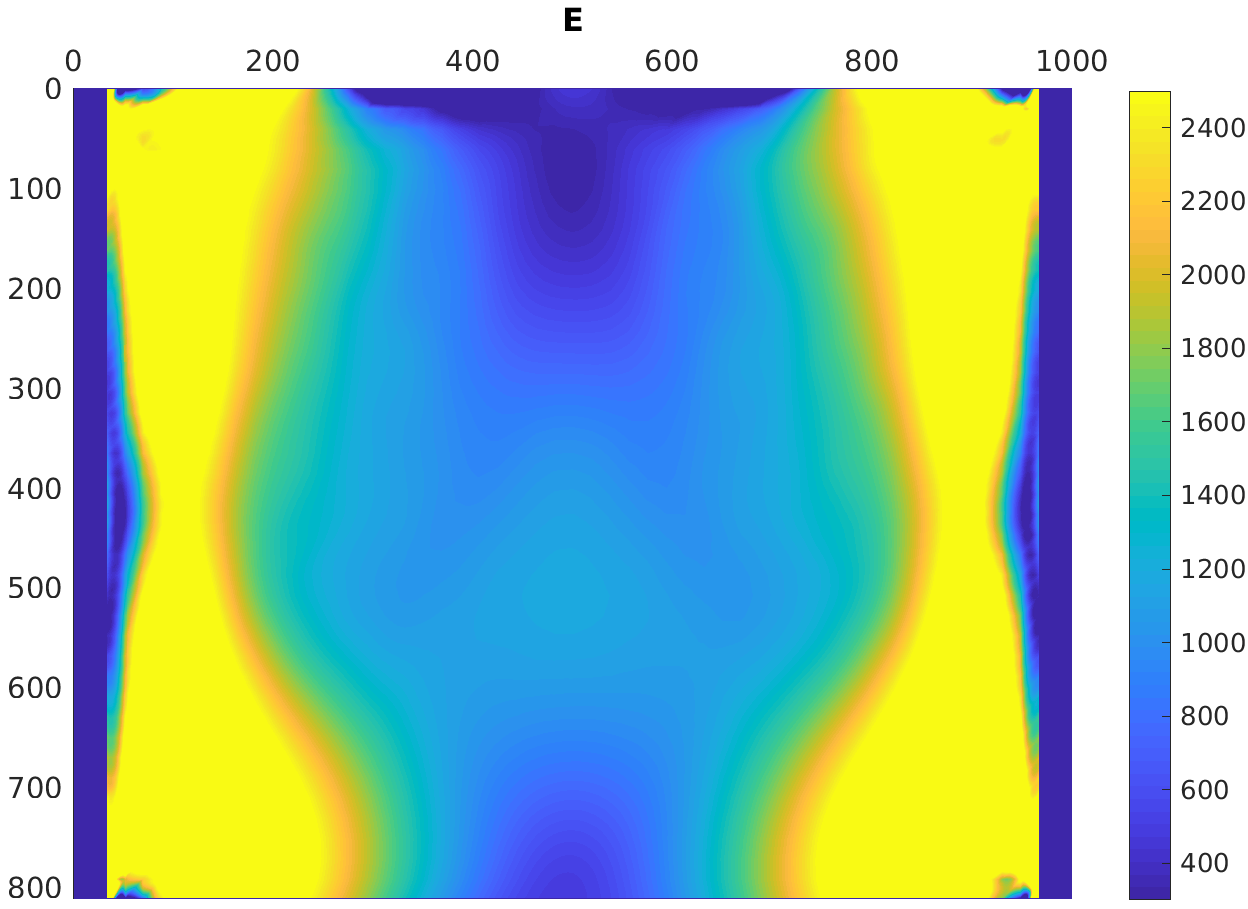}
	\caption{Reconstructed Lam\'e parameters $\lambda$ (left) and $\mu$ (middle), and the Young's modulus $E$ (right), given in kPA, for the experimental sample depicted in Figure~\ref{fig_OCT_speckle}. The results were obtained by applying the iterative reconstruction approach presented in Section~\ref{sect_it_reg_approach} using the operator $F$ to the estimated displacement field depicted in Figure~\ref{fig_OCT_displacement_alpha_comp} (right).}
	\label{fig_exp_rec}
\end{figure}

\section{Conclusion}\label{sect_conclusion}

We presented a displacement field estimation approach applicable to OCT images which utilizes the additional information of speckles present in the OCT images. Apart from a concise analysis of the proposed method, we provided a number of numerical examples demonstrating the usefulness of our approach, in particular when applied to obtain quantitative material parameter estimates in optical coherence elastography.

\section{Support}

The authors were funded by the Austrian Science Fund (FWF): F6807-N36 (ES and OS), project F6803-N36 (LK and WD), and project F6805-N36 (SH).

\section*{References}
\renewcommand{\i}{\ii}
\AtNextBibliography{\footnotesize} 
\printbibliography[heading=none]

@article{         AdiLiaKenJohSamBop2010,
  author        = {Adie, S. G. and Liang, X. and Kennedy, B. F. and John, R.
                   and Sampson, D. D. and Boppart, S. A.},
  title         = {Spectroscopic optical coherence elastography},
  journal       = jour-BOE,
  issn          = {1094-4087},
  language      = {eng},
  number        = {25},
  pages         = {25519--25534},
  url           = {http://search.proquest.com/docview/818645559/},
  volume        = {18},
  year          = {2010},
}

@article{         AlnBleHakJohKeh15,
  author        = {Aln{\ae}s, M.~S. and Blechta, J. and Hake, J. and Johansson,
                   A. and Kehlet, B. and Logg, A. and Richardson, C. and Ring,
                   J. and Rognes, M.~E. and Wells, G.~N.},
  title         = {The {FEniCS} Project Version 1.5},
  xdata         = {jour-ARNS},
  doi           = {10.11588/ans.2015.100.20553},
  number        = {100},
  pages         = {9--23},
  volume        = {3},
  year          = {2015},
  _doictrl      = {noauthor,nopages,nonumber},
}

@article{         AubDerKor99,
  author        = {Aubert, G. and Deriche, R. and Kornprobst, P.},
  title         = {Computing optical flow via variational techniques},
  xdata         = {jour-SJAM},
  issue         = {1},
  pages         = {156--182},
  volume        = {60},
  year          = {1999},
}

@article{         AubKor99,
  author        = {Aubert, G. and Kornprobst, P.},
  title         = {A mathematical study of the relaxed optical flow problem in
                   the space {$\textrm{BV}(\Omega)$}},
  xdata         = {jour-SJMA},
  number        = {6},
  pages         = {1282--1308},
  volume        = {30},
  year          = {1999},
}

@article{         BakSchaLewRotBla11,
  author        = {Baker, S. and Scharstein, D. and Lewis, J.~P. and Roth, S.
                   and Black, M.~J. and Szeliski, R.},
  title         = {A Database and Evaluation Methodology for Optical Flow},
  xdata         = {jour-IJCV},
  doi           = {10.1007/s11263-010-0390-2},
  month         = nov,
  number        = {1},
  pages         = {1--31},
  url           = {http://www.springerlink.com/index/10.1007/s11263-010-0390-2},
  volume        = {92},
  year          = {2011},
}

@article{         BlaAna96,
  author        = {Black, M.J. and Anandan, P.},
  title         = {The robust estimation of multiple motions: Parametric and
                   piecewise-smooth flow fields},
  xdata         = {jour-CVIU},
  pages         = {75--104},
  volume        = {63},
  year          = {1996},
}

@article{         BroMal11,
  author        = {Brox, T. and Malik, J.},
  title         = {Large Displacement Optical Flow: Descriptor Matching in
                   Variational Motion Estimation},
  xdata         = {jour-IEEETPAMI},
  doi           = {10.1109/tpami.2010.143},
  month         = mar,
  number        = {3},
  pages         = {500--513},
  publisher     = {Institute of Electrical and Electronics Engineers ({IEEE})},
  url           = {https://doi.org/10.1109/tpami.2010.143},
  volume        = {33},
  year          = {2011},
}

@article{         Doy12,
  author        = {Doyley, M.~M.},
  title         = {Model-based elastography: a survey of approaches to the
                   inverse elasticity problem},
  xdata         = {jour-PMB},
  pages         = {R35--R73},
  volume        = {57},
  year          = {2012},
}

@article{         DunKir01,
  author        = {Duncan, D.~D. and Kirkpatrick, S.~J.},
  title         = {Processing algorithms for tracking speckle shifts in optical
                   elastography of biological tissues},
  xdata         = {jour-JBO},
  doi           = {10.1117/1.1412224},
  number        = {4},
  pages         = {418},
  publisher     = {{SPIE}-Intl Soc Optical Eng},
  url           = {http://dx.doi.org/10.1117/1.1412224},
  volume        = {6},
  year          = {2001},
}

@article{         HubRam17,
  author        = {Hubmer, S. and Ramlau, R.},
  title         = {Convergence analysis of a two-point gradient method for
                   nonlinear ill-posed problems},
  xdata         = {jour-IP},
  doi           = {10.1088/1361-6420/aa7ac7},
  month         = aug,
  number        = {9},
  pages         = {095004},
  url           = {https://doi.org/10.1088/1361-6420/aa7ac7},
  volume        = {33},
  year          = {2017},
}

@article{         HubRam18,
  author        = {Hubmer, S. and Ramlau, R.},
  title         = {Nesterov's accelerated gradient method for nonlinear
                   ill-posed problems with a locally convex residual functional},
  xdata         = {jour-IP},
  doi           = {10.1088/1361-6420/aacebe},
  month         = jul,
  number        = {9},
  pages         = {095003},
  publisher     = {{IOP} Publishing},
  url           = {https://doi.org/10.1088/1361-6420/aacebe},
  volume        = {34},
  year          = {2018},
}

@article{         Jin16,
  author        = {Jin, Q.},
  title         = {Landweber-Kaczmarz method in Banach spaces with inexact
                   inner solvers},
  xdata         = {jour-IP},
  doi           = {10.1088/0266-5611/32/10/104005},
  month         = aug,
  number        = {10},
  pages         = {104005},
  url           = {https://doi.org/10.1088/0266-5611/32/10/104005},
  volume        = {32},
  year          = {2016},
}

@article{         KenKenSam2014,
  author        = {Kennedy, B. F. and Kennedy, K. M. and Sampson, D. D.},
  title         = {A Review of Optical Coherence Elastography: Fundamentals,
                   Techniques and Prospects},
  journal       = jour-IEEEJQE,
  issn          = {1077-260X},
  number        = {2},
  pages         = {272--288},
  publisher     = {IEEE},
  volume        = {20},
  year          = {2014},
}

@article{         Kin11,
  author        = {Kindermann},
  title         = {Convergence analysis of minimization-based noise level-free
                   parameter choice rules for linear ill-posed problems},
  xdata         = {jour-ETNA},
  pages         = {233--257},
  volume        = {38},
  year          = {2011},
}

@article{         KinRai19,
  author        = {Kindermann, S. and Raik, K.},
  title         = {Heuristic Parameter Choice Rules for Tikhonov Regularization
                   with Weakly Bounded Noise},
  xdata         = {jour-NFAO},
  doi           = {10.1080/01630563.2019.1604546},
  month         = apr,
  number        = {12},
  pages         = {1373--1394},
  url           = {https://doi.org/10.1080/01630563.2019.1604546},
  volume        = {40},
  year          = {2019},
}

@article{         KorDerAub99,
  author        = {Kornprobst, P. and Deriche, R. and Aubert, G.},
  title         = {Image sequence analysis via partial differential equations},
  xdata         = {jour-JMIV},
  pages         = {5--26},
  volume        = {11},
  year          = {1999},
}

@article{         LiWijDanSamMunKenObe2016,
  author        = {Li, D. and Wijesinghe, P. and Dantuono, J. T. and Sampson,
                   D. D. and Munro, P. R. T. and Kennedy, B. F. and Oberai, A.
                   A.},
  title         = {Quantitative Compression Optical Coherence Elastography as
                   an Inverse Elasticity Problem},
  journal       = jour-IEEEJQE,
  issn          = {1077-260X},
  language      = {eng},
  number        = {3},
  pages         = {277--287},
  publisher     = {IEEE},
  volume        = {22},
  year          = {2016},
}

@article{         LiWijSamKenMunObe2019,
  author        = {Li, D. and Wijesinghe, P. and Sampson, D. D. and Kennedy, B.
                   F. and Munro, P. R. T. and Oberai, A. A.},
  title         = {Volumetric quantitative optical coherence elastography with
                   an iterative inversion method.},
  journal       = jour-BOE,
  issn          = {2156-7085},
  language      = {eng},
  number        = {2},
  pages         = {384--398},
  url           = {http://search.proquest.com/docview/2185873238/},
  volume        = {10},
  year          = {2019},
}

@article{         ManOliDreMahKru01,
  author        = {Manduca, A. and Oliphant, T.~E. and Dresner, M.~A. and
                   Mahowald, J.~L. and Kruse, S. A. and Amromin, E. and
                   Felmlee, J.~P. and Greenleaf, J.~F. and Ehman, R.~L.},
  title         = {Magnetic resonance elastography: Non-invasive mapping of
                   tissue elasticity},
  xdata         = {jour-MIA},
  pages         = {237--354},
  volume        = {5},
  year          = {2001},
}

@article{         MeiSanKon13,
  author        = {Meinhardt-Llopis, E. and S{\'a}nchez P{\'e}rez, J. and
                   Kondermann, D.},
  title         = {Horn-Schunck Optical Flow with a Multi-Scale Strategy},
  xdata         = {jour-IMOL},
  doi           = {10.5201/ipol.2013.20},
  month         = jul,
  pages         = {151--172},
  publisher     = {Image Processing On Line},
  url           = {https://doi.org/10.5201/ipol.2013.20},
  volume        = {3},
  year          = {2013},
}

@article{         Nag87,
  author        = {Nagel, N.H.},
  title         = {On the estimation of optical flow: relations between new
                   approaches and some new results},
  xdata         = {jour-AI},
  pages         = {299--324},
  volume        = {33},
  year          = {1987},
}

@article{         PapBruBroDidWei06,
  author        = {Papenberg, N. and Bruhn, A. and Brox, T. and Didas, S. and
                   Weickert, J.},
  title         = {Highly Accurate Optic Flow Computation with Theoretically
                   Justified Warping},
  xdata         = {jour-IJCV},
  month         = jan,
  number        = {2},
  pages         = {141--158},
  url           = {http://www.springerlink.com/index/10.1007/s11263-005-3960-y},
  volume        = {67},
  year          = {2006},
}

@article{         RogPatFujBre04,
  author        = {Rogowska, J. and Patel, N.~A. and Fujimoto, J.~G. and
                   Brezinski, M.~E.},
  title         = {Optical coherence tomographic elastography technique for
                   measuring deformation and strain of atherosclerotic tissues},
  xdata         = {jour-H},
  doi           = {10.1136/hrt.2003.016956},
  month         = may,
  number        = {5},
  pages         = {556--562},
  publisher     = {{BMJ}},
  url           = {https://doi.org/10.1136/hrt.2003.016956},
  volume        = {90},
  year          = {2004},
  _doictrl      = {noauthor},
}

@article{         Schm98,
  author        = {Schmitt, J. M.},
  title         = {{OCT} elastography: imaging microscopic deformation and
                   strain of tissue},
  xdata         = {jour-OE},
  number        = {6},
  pages         = {199-211},
  volume        = {3},
  year          = {1998},
}

@article{         SchmXiaYun99,
  author        = {Schmitt, J. M. and Xiang, S. H. and Yung, K. M.},
  title         = {Speckle in Optical Coherence Tomography},
  xdata         = {jour-JBO},
  number        = {1},
  pages         = {95-105},
  volume        = {4},
  year          = {1999},
}

@article{         Schn91a,
  author        = {Schn{\"o}rr, Ch.},
  title         = {Determining optical flow for irregular domains by minimizing
                   quadratic functionals of a certain class},
  xdata         = {jour-IJCV},
  pages         = {25--38},
  volume        = {6},
  year          = {1991},
}

@article{         SunRotBla13,
  author        = {Sun, D. and Roth, S. and Black, Michael J.},
  title         = {A Quantitative Analysis of Current Practices in Optical Flow
                   Estimation and the Principles Behind Them},
  xdata         = {jour-IJCV},
  number        = {2},
  pages         = {115--137},
  volume        = {106},
  year          = {2013},
}

@article{         SunStaYan11,
  author        = {Sun, C. and Standish, B. and Yang, V. X. D.},
  title         = {Optical coherence elastography, current status and future
                   applications},
  xdata         = {jour-JBO},
  number        = {4},
  pages         = {043001},
  volume        = {16},
  year          = {2011},
}

@article{         WanLar15,
  author        = {Wang, S. and Larin, K.~V.},
  title         = {Optical coherence elastography for tissue characterization:
                   a review},
  xdata         = {jour-JBP},
  doi           = {10.1002/jbio.201400108},
  month         = nov,
  number        = {4},
  pages         = {279--302},
  volume        = {8},
  year          = {2015},
}

@book{          AubKor06,
  author      = {Aubert, G. and Kornprobst, P.},
  title       = {Mathematical problems in image processing},
  address     = {New York},
  edition     = {2},
  isbn        = {978-0387-32200-1},
  note        = {Partial differential equations and the calculus of variations,
                 With a foreword by Olivier Faugeras},
  pagetotal   = {xxxii+377},
  publisher   = {Springer},
  series      = {Applied Mathematical Sciences},
  volume      = {147},
  year        = {2006},
}

@book{          BauCom11,
  author      = {Bauschke, H.~H. and Combettes, P.~L.},
  title       = {Convex analysis and monotone operator theory in {H}ilbert
                 spaces},
  address     = {New York},
  doi         = {10.1007/978-1-4419-9467-7},
  publisher   = {Springer},
  series      = {CMS Books in Mathematics/Ouvrages de Math\'ematiques de la SMC},
  url         = {http://dx.doi.org/10.1007/978-1-4419-9467-7},
  year        = {2011},
}

@book{          Bra07,
  author      = {Braess, D.},
  title       = {Finite Elements},
  address     = {Cambridge},
  edition     = {3},
  note        = {Theory, Fast Solvers, and Applications in Solid Mechanics,
                 Translated from the 1992 German edition by L.L. Schumaker},
  publisher   = {Cambridge University Press},
  year        = {2007},
}

@book{          Cia94,
  author      = {Ciarlet, P.~G.},
  title       = {Mathematical Elasticity: Three-dimensional elasticity},
  isbn        = {9780444817761},
  lccn        = {87023741},
  number      = {1},
  series      = {Mathematical Elasticity},
  year        = {1994},
}

@book{          EngHanNeu96,
  author      = {Engl, H~.W. and Hanke, M. and Neubauer, A.},
  title       = {Regularization of inverse problems},
  address     = {Dordrecht},
  isbn        = {0-7923-4157-0},
  number      = {375},
  pagetotal   = {viii+321},
  publisher   = {Kluwer Academic Publishers Group},
  series      = {Mathematics and its Applications},
  year        = {1996},
}

@book{          McL00,
  author      = {McLean, W.},
  title       = {Strong Elliptic Systems and Boundary Integral Equations},
  address     = {London},
  publisher   = {Cambridge University Press},
  year        = {2000},
}

@book{          Mod03,
  author      = {Modersitzki, J.},
  title       = {Numerical Methods for Image Registration},
  address     = {New York},
  publisher   = {Oxford University Press},
  year        = {2003},
}

@book{          Mod09,
  author      = {Modersitzki, J.},
  title       = {F{AIR}: flexible algorithms for image registration},
  address     = {Philadelphia, PA},
  doi         = {10.1137/1.9780898718843},
  publisher   = {Society for Industrial and Applied Mathematics (SIAM)},
  series      = {Fundamentals of Algorithms},
  url         = {http://dx.doi.org/10.1137/1.9780898718843},
  volume      = {6},
  year        = {2009},
}

@book{          Nec11,
  author      = {Necas, J.},
  title       = {Direct Methods in the Theory of Elliptic Equations},
  isbn        = {9783642104558},
  publisher   = {Springer Berlin Heidelberg},
  series      = {Springer Monographs in Mathematics},
  year        = {2011},
}

@incollection{WijKenSam2020,
  author = {Wijesinghe, P. and Kennedy B. F. and Sampson D. D.},
  title = {Chapter 9 - Optical elastography on the microscale},
  editor = {Alam, S. K. and Garra, B. S.},
  booktitle = {Tissue Elasticity Imaging},
  publisher = {Elsevier},
  address = {Amsterdam},
  pages = {185--229},
  year = {2020},
  isbn = {978-0-12-809661-1},
  doi = {https://doi.org/10.1016/B978-0-12-809661-1.00009-1},
  url = {http://www.sciencedirect.com/science/article/pii/B9780128096611000091}
}

@incollection{    WeiBruBroPap06,
  author        = {Weickert, J. and Bruhn, A. and Brox, T. and Papenberg, N.},
  title         = {A survey on variational optic flow methods for small
                   displacements},
  address       = {Berlin Heidelberg},
  booktitle     = {Mathematical Models for Registration and Applications to
                   Medical Imaging},
  doi           = {10.1007/978-3-540-34767-5_5},
  editor        = {Scherzer, O.},
  isbn          = {978-3-540-25029-6},
  pages         = {103--136},
  publisher     = {Springer},
  series        = {Mathematics in Industry},
  url           = {http://dx.doi.org/10.1007/978-3-540-34767-5_5},
  volume        = {10},
  year          = {2006},
}

@article{Sch96,
  author     = {Scherzer, O.},
  title      = {A convergence analysis of a method of steepest descent and a two-step algorithm for nonlinear ill-posed problems},
  coden      = {NFADOL},
  doi        = {10.1080/01630569608816691},
  issn       = {0163-0563},
  journal    = jour-NFAO,
  fjournal   = fjour-NFAO,
  xdata      = {jour-NFAO},
  number     = {1-2},
  pages      = {197--214},
  url        = {http://dx.doi.org/10.1080/01630569608816691},
  volume     = {17},
  year       = {1996},
}

@article{GlaSchWid15,
  author     = {Glatz, T. and Scherzer, O. and Widlak, T.},
  title      = {Texture Generation for Photoacoustic Elastography},
  doi        = {10.1007/s10851-015-0561-4},
  journal    = jour-JMIV,
  fjournal   = fjour-JMIV,
  xdata      = {jour-JMIV},
  month      = jan,
  number     = {3},
  pages      = {369--384},
  volume     = {52},
  year       = {2015},
  openarch   = {app_springer_12m},
}

@inproceedings{SchmZabWidGlaLiu15,
  author     = {Schmid, J. and Zabihian, B. and Widlak, T. and Glatz, T. and Liu, M. and Drexler, W. and Scherzer, O.},
  title      = {Texture generation in compressional photoacoustic elastography},
  booktitle  = {Photons Plus Ultrasound: Imaging and Sensing 2015},
  doi        = {10.1117/12.2079672},
  pages      = {93232S},
  series     = {Proceedings of SPIE},
  volume     = {9323},
  year       = {2015},
}

@article{WidSch15,
  author     = {Widlak, T. and Scherzer, O.},
  title      = {Stability in the linearized problem of quantitative elastography},
  doi        = {10.1088/0266-5611/31/3/035005},
  journal    = jour-IP,
  fjournal   = fjour-IP,
  xdata      = {jour-IP},
  number     = {3},
  pages      = {035005},
  url        = {http://iopscience.iop.org/article/10.1088/0266-5611/31/3/035005/pdf},
  volume     = {31},
  year       = {2015},
  openarch   = {gold_open_access},
}

@article{HubSheNeuSch18,
  author     = {Hubmer, S. and Sherina, E. and Neubauer, A. and Scherzer, O.},
  title      = {Lam{\'e} Parameter Estimation from Static Displacement Field Measurements in the Framework of Nonlinear Inverse Problems},
  doi        = {10.1137/17M1154461},
  journal    = jour-SIS,
  fjournal   = fjour-SIS,
  xdata      = {jour-SIS},
  number     = {2},
  pages      = {1268-1293},
  volume     = {11},
  year       = {2018},
  openarch   = {pv_siam},
}

@book{KalNeuSch08,
  author     = {Kaltenbacher, B. and Neubauer, A. and Scherzer, O.},
  title      = {Iterative regularization methods for nonlinear ill-posed problems},
  address    = {Berlin},
  doi        = {10.1515/9783110208276},
  isbn       = {978-3-11-020420-9},
  publisher  = {Walter de Gruyter},
  series     = {Radon Series on Computational and Applied Mathematics},
  url        = {http://dx.doi.org/10.1515/9783110208276},
  volume     = {6},
  year       = {2008},
}

@inproceedings{  BroBreMal09,
  author      = {Brox, T. and Bregler, C. and Malik, J.},
  title       = {Large displacement optical flow},
  booktitle   = {2009 {IEEE} Conference on Computer Vision and Pattern
                 Recognition},
  doi         = {10.1109/cvpr.2009.5206697},
  month       = jun,
  publisher   = {{IEEE}},
  url         = {https://doi.org/10.1109/cvpr.2009.5206697},
  year        = {2009},
}

@inproceedings{  LauKorMem04,
  author      = {F. Lauze and P. Kornprobst and E. Memin},
  title       = {{A Coarse to Fine Multiscale Approach for Linear Least Squares
                 Optical Flow Estimation}},
  booktitle   = {British Machine Vision Conference},
  pages       = {767--776},
  year        = {2004},
}

@inproceedings{  Sny89,
  author      = {Snyder, M.~A.},
  title       = {On the mathematical foundations of smoothness constraints for
                 the determination of optical flow and for surface
                 reconstruction},
  booktitle   = {Proc. Workshop Visual Motion},
  pages       = {107--115},
  year        = {1989},
}

@XDATA{jour-AI,
  journaltitle	= {Artificial Intelligence},
  shortjournal	= {Artificial Intelligence},
}

@XDATA{jour-ARNS,
  journaltitle = {Archive of Numerical Software},
}

@XDATA{jour-CVIU,
  journaltitle	= {Computer Vision and Image Understanding},
  shortjournal	= {Comput. Vision Image Understanding},
  issn 		= {1077-3142},
}

@XDATA{jour-ETNA,
  journaltitle	= {Electronic Transactions on Numerical Analysis},
  shortjournal	= {Electron. Trans. Numer. Anal.},
}

@XDATA{jour-H,
  journaltitle	= {Heart},
  shortjournal	= {Heart},
}

@XDATA{jour-IEEETPAMI,
  journaltitle	= {IEEE Transactions on Pattern Analysis and Machine Intelligence},
  shortjournal	= {IEEE Trans. Pattern Anal. Mach. Intell.},
  issn		= {0162-8828},
}

@XDATA{jour-IJCV,
  journaltitle	= {International Journal of Computer Vision},
  shortjournal	= {Int. J. Comput. Vision},
  issn		= {0920-5691},
  publisher	= {Springer},
  location	= {Netherlands}
}

@XDATA{jour-IMOL,
  journaltitle	= {Image Processing On Line},
  shortjournal	= {Image Proc. On Line},
}

@XDATA{jour-IP,
  journaltitle	= {Inverse Problems},
  shortjournal	= {Inverse Probl.},
  issn		= {0266-5611},
}

@XDATA{jour-JBO,
  journaltitle	= {Journal of Biomedical Optics},
  shortjournal	= {J. Biomed. Opt.},
  issn		= {1083-3668},
  publisher	= {SPIE},
}

@XDATA{jour-JBP,
  journaltitle	= {Journal of Biophotonics},
  shortjournal	= {J. Biophotonics},
  issn		= {1864-0648},
  publisher	= {Wiley},
}

@XDATA{jour-JMIV,
  journaltitle	= {Journal of Mathematical Imaging and Vision},
  shortjournal	= {J. Math. Imaging Vision},
  issn 		= {0924-9907},
  publisher	= {Springer},
  location	= {Netherlands},
}

@XDATA{jour-MIA,
  journaltitle	= {Medical Image Analysis},
  shortjournal	= {Med. Image Anal.},
  issn		= {1361-8415},
}

@XDATA{jour-NFAO,
  journaltitle	= {Numerical Functional Analysis and Optimization},
  shortjournal	= {Numer. Funct. Anal. Optim.},
  issn		= {0163-0563},
}

@XDATA{jour-OE,
  journaltitle	= {Optics Express},
  shortjournal	= {Opt. Express},
  publisher = {OSA},
}

@XDATA{jour-PMB,
  journaltitle	= {Physics in Medicine and Biology},
  shortjournal	= {Phys. Med. Biol.},
  issn		= {0031-9155},
  publisher	= {IOP Publishing Ltd},
}

@XDATA{jour-SIS,
  journaltitle	= {SIAM Journal on Imaging Sciences},
  shortjournal	= {SIAM J. Imaging Sciences},
  issn		= {1936-4954},
}

@XDATA{jour-SJAM,
  journaltitle	= {SIAM Journal on Applied Mathematics},
  shortjournal	= {SIAM J. Appl. Math.},
  issn		= {0036-1399},
}

@XDATA{jour-SJMA,
  journaltitle	= {SIAM Journal on Mathematical Analysis},
  shortjournal	= {SIAM J. Math. Anal.},
  issn		= {0036-1410},
}

@PREAMBLE{"\def\cprime{$'$} "}

@PREAMBLE{ {\providecommand{\noopsort}[1]{}} }

@STRING{jour-AI = {Artificial Intelligence}}

@STRING{jour-ARNS = {Archive of Numerical Software}}

@STRING{jour-BOE = {Biomed. Opt. Express}}

@STRING{jour-CVIU = {Comput. Vision Image Understanding}}

@STRING{jour-ETNA = {Electron. Trans. Numer. Anal.}}

@STRING{jour-H = {Heart}}

@STRING{jour-IEEEJQE = {IEEE J. Quantum Electron.}}

@STRING{jour-IEEETPAMI = {IEEE Trans. Pattern Anal. Mach. Intell.}}

@STRING{jour-IJCV = {Int. J. Comput. Vision}}

@STRING{jour-IMOL = {Image Proc. On Line}}

@STRING{fjour-IP = {Inverse Problems}}

@STRING{jour-IP = {Inverse Probl.}}

@STRING{jour-JBO = {J. Biomed. Opt.}}

@STRING{jour-JBP = {J. Biophotonics}}

@STRING{fjour-JMIV = {Journal of Mathematical Imaging and Vision}}

@STRING{jour-JMIV = {J. Math. Imaging Vision}}

@STRING{jour-MIA = {Med. Image Anal.}}

@STRING{fjour-NFAO = {Numerical Functional Analysis and Optimization}}

@STRING{jour-NFAO = {Numer. Funct. Anal. Optim.}}

@STRING{jour-OE = {Opt. Express}}

@STRING{jour-PMB = {Phys. Med. Biol.}}

@STRING{fjour-SIS = {SIAM Journal on Imaging Sciences}}

@STRING{jour-SIS  = {SIAM J. Imaging Sciences}}

@STRING{jour-SJAM = {SIAM J. Appl. Math.}}

@STRING{jour-SJMA = {SIAM J. Math. Anal.}}

\end{document}